\documentclass[12pt]{article}

\usepackage{amsmath, amsthm, amssymb}
\usepackage{graphicx}
\usepackage{psfrag}
\usepackage{epsfig}
\usepackage{epstopdf}
\usepackage{color}

\if 1 = 0
\usepackage{graphicx}
\usepackage[dvips]{graphicx}
\usepackage{color}
\fi

\usepackage{bm}
\usepackage{tikz}
\usepackage{multirow}
\usetikzlibrary{arrows,backgrounds,snakes}

\newtheoremstyle{plainNoItalics}{}{}{\normalfont}{}{\bfseries}{.}{ }{}

\theoremstyle{plain}
\newtheorem{thm}{Theorem}[section]

\theoremstyle{plainNoItalics}

\newtheorem{defn}[thm]{Definition}
\newtheorem{rem}[thm]{Remark}
\newtheorem{prop}[thm]{Proposition}

\newcommand{\beq}{\begin{equation}}
\newcommand{\eeq}{\end{equation}}
\newcommand{\beqa}{\begin{eqnarray}}
\newcommand{\eeqa}{\end{eqnarray}}
\newcommand{\bit}{\begin{itemize}}
\newcommand{\eit}{\end{itemize}}
\newcommand{\bedef}{\begin{defn}}
\newcommand{\edefn}{\end{defn}}
\newcommand{\bpro}{\begin{prop}}
\newcommand{\epro}{\end{prop}}

\begin{document}

\baselineskip=1.8pc


\begin{center}
{\bf
Runge-Kutta Discontinuous Galerkin Method for Traffic Flow Model on Networks
}
\end{center}

\vspace{.2in}
\centerline{
Suncica Canic
\footnote{Department of Mathematics, University of Houston,
Houston, TX, 77204. E-mail: canic@math.uh.edu.
The research of the first author is partially supported by NSF under grants DMS-1263572,  DMS-1318763,  DMS-1311709,
DMS-1262385, and DMS-1109189.
},
Benedetto Piccoli
\footnote{Department of Mathematics and Center for Computational and Integrative Biology,
Rutgers University - Camden, Camden, NJ, 08102. E-mail: piccoli@camden.rutgers.edu.
The research of the second author is partially supported by NSF under grant DMS-1107444.
},
Jing-Mei Qiu \footnote{Department of Mathematics, University of Houston,
Houston, TX, 77204. E-mail: jingqiu@math.uh.edu.
The research of the third and the fourth author is partially supported by Air Force Office of Scientific Computing YIP grant FA9550-12-0318, NSF grant DMS-1217008 and University of Houston.},
Tan Ren \footnote{School of Aerospace Engineering, Beijing Institute of Technology, Beijing, 100081. E-mail: rentanx@gmail.com
}
}

\bigskip
\noindent
{\bf Abstract.}
We propose a bound-preserving Runge-Kutta (RK) discontinuous Galerkin (DG) method as an efficient, effective and compact numerical approach  for numerical simulation of traffic flow problems on networks, with arbitrary high order accuracy.
Road networks are modeled by graphs, composed of a finite number of roads that meet at junctions. On each road, a scalar conservation law describes
the dynamics, while coupling conditions are specified at
junctions to define flow separation or convergence at the points where roads meet. We incorporate such coupling conditions in the RK DG framework, and apply an arbitrary high order bound preserving limiter to the RK DG method to preserve the physical bounds on the network solutions (car density). We showcase the proposed algorithm
on several benchmark test cases from the literature, as well as several new challenging examples with rich solution structures.
{Modeling and simulation of Cauchy problems for traffic flows on networks is notorious for lack of uniqueness or (Lipschitz) 
continuous dependence. The discontinuous Galerkin method proposed here deals elegantly with these problems, and is perhaps
the only realistic and efficient high-order method for network problems.}

\vfill


\newpage

\newpage

\section{Introduction}
\label{sec1}
\setcounter{equation}{0}
\setcounter{figure}{0}
\setcounter{table}{0}

In this paper we deal with vehicular traffic models on networks.
More precisely, we focus on the classical Lighthill-Whitham-Richards model
(see~\cite{LighthillWhitham1955aa,Richards1956aa}),
which consists of a single conservation laws for the car density.
The model describes the evolution of traffic load on a single road,
assuming that the average velocity depends only on the density via
a closure relation. The resulting density-flow function is usually
called {\em fundamental diagram} in engineering literature.
Such model was adapted to networks in a number of different ways~\cite{HoldenRisebro1995aa,LebacqueKhoshyaran2004aa,CocliteGaravelloPiccoli2005aa} 
depending on the rules used to describe the dynamics at junctions.
The only conservation of cars is not sufficient to isolate
a unique dynamics, thus additional rules, such as traffic distribution matrices,
are to be prescribed. 
In particular, various authors proposed a set of additional rules
which isolate a unique solution for every Riemann problem at a junction, i.e. a Cauchy
problem with initial density constant on each road. In that case the map providing a unique
solution to Riemann problems is called Riemann solver.
A fairly general theory for such models on networks is now available, see~\cite{GaravelloPiccoli2006ab,MR2566716}.

It is interesting to notice that lack of continuous dependence (or uniqueness)
may indeed happen for Cauchy problems even if we do have unique solutions to Riemann problems.
More precisely, the phenomenon of lack of Lipschitz continuous dependence is illustrated
in Section 5.4 of the book \cite{GaravelloPiccoli2006ab}.
However, some Riemann solvers do provide Lipschitz continuous dependence for Cauchy problems
(and thus also uniqueness). Examples can be found in
\cite{MR2262939}, \cite{GaravelloPiccoli2006ab} (Chapter 9) and \cite{MR2566716}.
The specific solvers considered in this paper fall in this category, except for the case
of two incoming and two outgoing roads (for which Lipschitz continuous dependence is false
but uniqueness and continuous dependence are still open problems.)

Due to limitations of the single conservation law to describe dynamics
in case of congestion, various models consisting of two equations
(conservation of car mass and balance of ``momentum")
have been proposed, see e.g.~\cite{AwRascle2000aa,ColomboGoatin2006aa}.
Numerical methods for conservation laws on networks were developed
mainly based on first order schemes, see~\cite{bretti2006numerical,CutoloPiccoliRarita2011,HertyKlar2003aa}.
First order schemes on networks have the same limitations as when they are applied to problems defined on a single real line:
weak solutions are not well approximated, unless the spatial mesh is very fine to resolve solution structures. 
For this reason, we propose to use Discontinuous Galerkin methods with arbitrary high-order
accuracy, which will be adapted in this paper to graph domains. Adaptation to graph problems requires supplementing the classical 
DG method with coupling conditions that hold at graph's vertices. We propose the use of Runge-Kutta DG methods with total variation bounded limiters as a straight forward way of implementing the coupling conditions, while preserving the upper and lower bounds of DG solutions with a bound preserving limiter.

Since the late 80s, DG methods have been gaining great popularity as methods of choice for solving systems of hyperbolic conservation laws, with high order accuracy for smooth solutions and good shock capturing capabilities. We refer the reader to review papers and books~\cite{cockburn2001runge, cockburn2000development} for the history, development, and applications of the methods. 
The high order accuracy in time evolution is realized by applying the strong stability preserving (SPP) Runge-Kutta (RK) time discretization via the method-of-line approach. Compared with the existing high order finite volume and finite difference schemes, DG methods are more flexible with general meshes and local approximations, hence more suitable for h-p adaptivity. They are very compact in the sense that the update of the solution on one element only depends on direct neighboring elements, thus allowing for easy handling of various boundary conditions with high order accuracy and great parallel efficiency. Compared with the classical continuous finite element methods, DG methods are advantageous in capturing solutions with discontinuities or sharp gradients for convection dominant problems. 

We propose to use the high order RK DG method with total variation bounded limiters as a general approach for simulating hyperbolic network problems. The compactness of the DG method enables a straightforward way of implementing coupling conditions at junctions. The bound preserving property of a first order monotone scheme for our traffic flow model on networks is theoretically proved, thanks to the rule of maximizing fluxes at junctions. Such property enables the application of bound preserving limiters, while maintaining classical high order accuracy of the RK DG method. Numerical results on benchmark problems from the literature, as well as on the ones with rich solution structures that we constructed in this paper, showcase the effectiveness of the proposed approach. 
We emphasize that, to our best knowledge, the DG method perhaps is the only realistic and efficient high order method for network problems. Existing high order finite difference and finite volume schemes would involve a wide and one-sided stencil in reconstructing solutions at junctions; such one-sided reconstruction stencil would lead to potential accuracy and stability issues. This paper is an initial step in applying the DG method to network problems; further development of the method to nonlinear hyperbolic systems with additional challenges in resolving junction conditions and numerical stability will be subject to future investigation.

The paper is organized as follows. Section~\ref{sec2} is on the background of traffic flow models on networks, with a general description of coupling conditions at junctions. Section~\ref{sec3} presents the proposed high order Runge-Kutta discontinuous Galerkin method for network problems with bound preserving properties. Section~\ref{sec4} demonstrates the performance of the proposed schemes on benchmark test problems from the literature and in challenging test cases with rich solution structures. Finally, a conclusion is given in Section~\ref{conclusion}.

\section{{Background on traffic flow models on networks}}
\label{sec2}
\setcounter{equation}{0}
\setcounter{figure}{0}
\setcounter{table}{0}

The nonlinear traffic model based on conservation of cars is a scalar hyperbolic conservation law in the form of
\begin{equation}
 \partial_t\rho+\partial_x f(\rho)=0,
 \label{trafficeq}
\end{equation}
where $\rho=\rho(t,x)\in [0,\rho_{max}]$ is the density of cars, with $\rho_{max}$ being the maximum density of cars on the road; $f(\rho)=\rho v(t,x)$ is the flux. The main assumption of this model
is that the average velocity $v$ is a function depending only on the density $\rho$,
thus giving rise to~(\ref{trafficeq}). The usual assumptions on
$f$ is that $f(0)=f(\rho_{max})=0$ and that $f$ is strictly concave,
thus has a unique maximum point $\sigma$ called the critical density.
Indeed, below $\sigma$ the traffic is said to be in free flow
and the flux $f$ is an increasing function of the density. On the other side,
above $\sigma$ the flux is a decreasing function of the density,
representing congestion.\\
For future use, we define:
\begin{defn}
Let $\tau :\left[ 0,\rho _{\max }\right] \rightarrow \left[ 0,\rho
_{\max }\right] $ be the map such that $f\left( \tau \left( \rho
\right) \right) =f\left( \rho \right) $ for every $\rho \in \left[
0,\rho _{\max }\right] ,$ and $\tau \left( \rho \right) \neq \rho
$ for every $\rho \in \left[ 0,\rho _{\max }\right] \backslash
\left\{ \sigma \right\} .$
\end{defn}

A network is described by a topological graph, i.e. a couple $\left(
\mathcal{I},\mathcal{J}\right) , $ where $\mathcal{I=}\left\{
I_{i}:i=1,...,N\right\} $ is a collection of intervals representing roads,
and  $\mathcal{J}$ is a collection of vertices representing the junctions.
For a fixed junction $J$, a Riemann Problem (RP) is a Cauchy Problem with initial data which are constant on each 
road incident at the junction.
The evolution on the whole network of the solution to~(\ref{trafficeq})
is determined once one assigns a Riemann Solver at each junction, i.e. a map
assigning a solution to every Riemann Problem at the junction.
More precisely, given initial conditions $(\rho_{i,0}, \rho_{j,0})$, where $i$
ranges over incoming roads and $j$ over outgoing ones, we will assign
density values $(\widehat\rho_{i}, \widehat\rho_{j})$ so that the solution
on the incoming road $i$ is given by a single wave 
$(\rho_{i,0},\widehat\rho_{i})$, and on the outgoing road $j$
by the single wave $(\widehat \rho_{j},\rho_{j,0})$.

We consider the Riemann Solver based on the following rules:
\begin{description}
\item[(A)] There exists traffic distribution coefficients 
$\alpha_{ji}\in ]0,1[$, representing the portion of traffic
from incoming road $i$ going to outgoing road $j$.
The resulting traffic distribution matrix:
\begin{equation*}
A=\left\{ \alpha _{ji}\right\} _{j=n+1,...,n+m,\text{
}i=1,...,n}\in \mathbb{R}^{m\times n},
\end{equation*}
is row stochastic, i.e. for every $i$ it holds:
\begin{equation*}
\sum_{j}\alpha _{ji}=1.
\end{equation*}%
\item[(B)] Respecting (A), drivers behave so as to maximize the
flux through the junction. In other words the sum of the flux over
incoming roads is maximized.
\end{description}
If $n>m$ a yielding rule, (C), is needed. 
\begin{description}
\item[(C)] For example, consider the case of two incoming roads $a$ and $b$ and one outgoing road $c$. Assume that not all cars can enter the road $c$, and let
$Q$ be the amount that can do it. Then, $q Q$ cars come from the road
$a$ and $\left( 1-q \right) Q$ cars from the road $b$.
\end{description}

Now we describe the solutions generated at junctions using rules
(A), (B) and (C). \\
Notice that solving a Riemann Problem at a junction is equivalent
to solving Initial Boundary Value Problems (IBVP) on each road.
Since solutions to IBVP may not attain the boundary values,
due to the nonlinearity of the equation, one has to impose
admissible values on each road, which generate only waves
with negative speed on incoming roads and positive on outgoing ones.
Indeed, if waves would enter the junction, then the solution to the
IBVP may not attain the prescribed boundary value and, for instance,
even violate conservation of cars, see also~\cite{GaravelloPiccoli2006ab}.
In turn, this allows to state the problem in terms of fluxes,
since densities can be reconstructed due to these restrictions.
Moreover, we have some bounds on maximal flows on each road,
more precisely we have:
\begin{prop}\label{prop:flux-limit}
Let $\left( \rho _{1,0},\rho _{2,0},...,\rho
_{n+m,0}\right) $ be the initial densities of a RP at  $J$ and
$\gamma _{i }^{\max },$ $i =1,...,n$ and $\gamma
_{j}^{\max },$ $j =n+1,...,n+m$ be the maximum fluxes that
can be obtained on incoming roads and outgoing roads,
respectively. Then:
\begin{equation}
\gamma _{i }^{\max }=\left\{
\begin{tabular}{ll}
$f\left( \rho _{i ,0}\right) ,$ & if $\rho _{i ,0}\in
\left[
0,\sigma \right] ,$ \\
$f\left( \sigma \right) ,$ & if $\rho _{i ,0}\in \left]
\sigma ,\rho _{\max } \right] ,$
\end{tabular}
\right. i =1,..,n,  \label{fluxUno}
\end{equation}
\begin{equation}
\gamma _{j }^{\max }=\left\{
\begin{tabular}{ll}
$f\left( \sigma \right) ,$ & if $\rho _{j ,0}\in \left[
0,\sigma \right] ,
$ \\
$f\left( \rho _{j,0}\right) ,$ & if $\rho _{j ,0}\in \left]
\sigma ,\rho _{\max } \right] ,$
\end{tabular}
\right. j =n+1,..,n+m.  \label{fluxDue}
\end{equation}
In particular, densities can be recontructed by flows at the junction.
\end{prop}

\begin{proof}
Consider first an incoming road $i$ and indicate by $\widehat\rho_i$
the trace at the junction for positive times. Then the IBVP on road
$i$ is solved by a single wave:  $\left( \rho _{i,0}, \widehat{\rho }_{i}\right) $,
which must have negative speed.
If $\rho_{i,0}\in \left[ 0,\sigma \right] ,$ then $\widehat{\rho }_{i}$
either is $\rho _{i,0}$ or belongs to $\left] \tau \left( \rho
_{i,0}\right) ,1\right] .$ In the first case, there is no wave,
while in the second case the wave $\left( \rho
_{i,0},\widehat{\rho }_{i}\right) $ is a shock wave with negative
speed, see Figure~\ref{Figure theo4Uno} (left).
Therefore the maximal flux is given by $f(\rho_{i,0})$.
Moreover, there exists a unique value of $\widehat\rho_i$,
which is compatible with a given value of the flux
in the interval $[0,f(\rho_{i,0})]$.

If, instead, $\rho _{i,0}\in \left[ \sigma ,1\right],$ then
$\widehat{\rho } _{i}\in \left[ \sigma ,1\right] $ and the wave
$\left( \rho _{i,0},\widehat{ \rho }_{i}\right) $ is a rarefaction
or a shock wave with negative speed, see Figure~\ref{Figure theo4Uno} (right). 
In this case the maximal flux is given by $f(\sigma)$ and, again,
there exists a unique value of $\widehat\rho_i$,
which is compatible with a given value of the flux
in the interval $[0,f(\sigma)]$.

For an outgoing road, the analysis is analogous, see Figure~\ref{Figure theo4Due} . 
\end{proof}

\begin{figure}[tbph]
\includegraphics[height=1.9in, width=6.2in
]{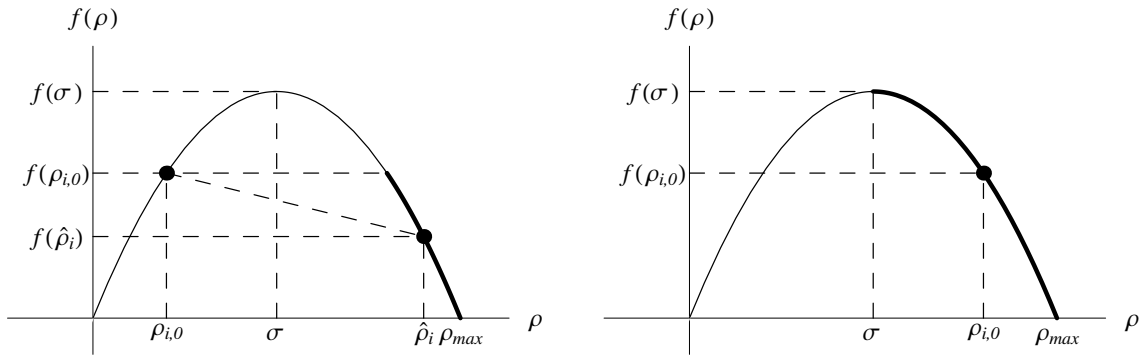} \caption{Images of Riemann solvers for the incoming
roads.} \label{Figure theo4Uno}
\end{figure}

\begin{figure}[tbph]
\includegraphics[height=1.9in, width=6.2in
]{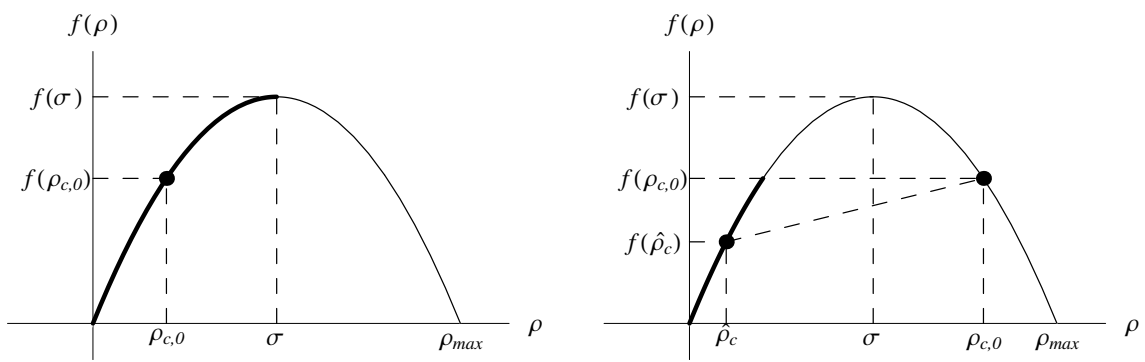} \caption{Images of Riemann solvers for the
outgoing road.} \label{Figure theo4Due}
\end{figure}

Proposition~\ref{prop:flux-limit} allows to restate rules (A), (B)
and (C) as a Linear Programming problem in terms of the incoming fluxes
$\widehat\gamma_i=f(\widehat\rho_i)$. Indeed, rule (A) allows to determine
the outgoing fluxes $\widehat\gamma_j=f(\widehat\rho_j)$ in terms of the incoming ones.
Then rule (B) provides a linear functional in the fluxes $\widehat\gamma_i$
to be maximized. The constraints are given by the formulas~(\ref{fluxUno}) and~(\ref{fluxDue}). Rule (C) allows to choose a unique
solution to the Linear Programming problem in case
of more incoming than outgoing roads.

In the following sections we will explicitely solve the Riemann Problems
in the following cases: junctions of type $2\times 1$ (two incoming roads and one outgoing road), junctions of type $1\times 2$ (one incoming road and two outgoing roads), and junctions of type $2\times 2$ (two incoming roads and two outgoing roads).  We refer the reader to~\cite{GaravelloPiccoli2006ab} for a complete description of the general case.

\subsection{The case of $n$ = 2 incoming roads and $m$ = 1 outgoing road}
\label{sec: 2.1}

Let us consider the junction with two incoming roads
$a$ and $b$ and one outgoing road $c$.
Given initial data $\left( \rho _{a,0},\rho _{b,0},\rho
_{c,0}\right) $ we construct a solution in the following way. 
To maximize the through traffic (rule (B)), we set:
\[
\widehat{\gamma }_{c}=\min \left\{ \gamma _{a}^{\max }+\gamma
_{b}^{\max },\gamma _{c}^{\max }\right\} ,
\]
where $\gamma _{i}^{\max },$ $i=a,b,$ is defined as in~(\ref{fluxUno}) and $%
\gamma _{c}^{\max }$ as in~(\ref{fluxDue}). Notice that in this
case the matrix $A$ (or rule (A)) is simply given by the column
vector $(1,1)$, thus it gives no additional restriction.\\
Consider now the space $\left( \gamma _{a},\gamma _{b}\right) $ and the line:
\begin{equation}
\gamma _{b}=\frac{1-q}{q}\gamma _{a},  \label{primaRetta}
\end{equation}
defined according to rule (C).   Let $P$ be the point
of intersection of the line~(\ref{primaRetta}) with the line
$\gamma _{a}+\gamma _{b}=\widehat{\gamma }_{c}.$ The final fluxes
must belong to the region%
\[
\Omega =\left\{ \left( \gamma _{a},\gamma _{b}\right) :0\leq
\gamma _{i}\leq
\gamma _{i}^{\max },\text{ }0\leq \gamma _{a}+\gamma _{b}\leq \widehat{%
\gamma }_{c},\text{ }i=a,b\right\} .
\]%
There are two different cases:
\begin{enumerate}
\item $P$ belongs to $\Omega ;$

\item $P$ does not belong to $\Omega .$
\end{enumerate}
The two cases are represented in Figure~\ref{Riemann1}.
\begin{figure}[tbph]
\begin{center}
\includegraphics{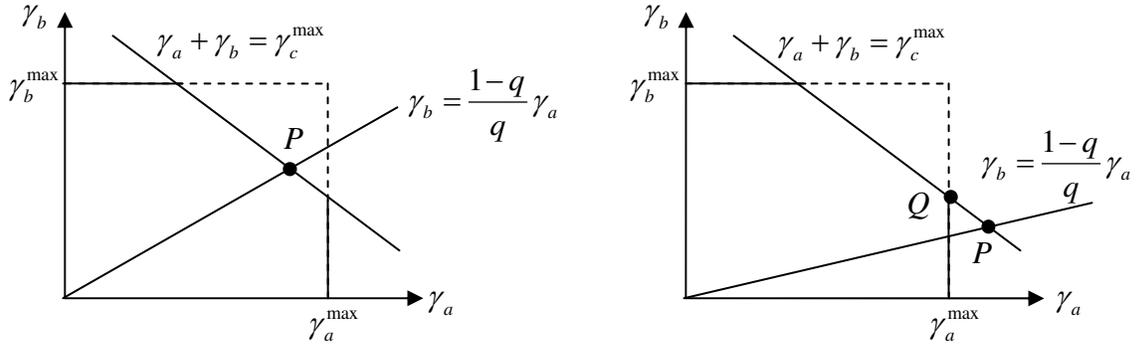}
\end{center}
\caption{The cases 1) and 2).  \label{Riemann1}}
\end{figure}
In the first case, we set $\left( \widehat{\gamma }_{a},\widehat{\gamma }%
_{b}\right) =P,$ while in the second case we set $\left( \widehat{\gamma }%
_{a},\widehat{\gamma }_{b}\right) =Q,$ where $Q$ is the point of
$\Omega \cap \left\{ \left( \gamma _{a},\gamma _{b}\right) :\text{
}\gamma
_{a}+\gamma _{b}=\widehat{\gamma }_{c}\right\} $ closest to the line~(\ref{primaRetta}).  Once we have determined $\widehat{\gamma }_{a}$
and $\widehat{\gamma }_{b}$ (and $\widehat{\gamma }_{c}$), we can
find in a unique way $\widehat{\rho }_{i}$, $i\in \left\{
a,b,c\right\} $: 
\begin{thm}\label{theo:rho-values}
Consider a junction $J$ with $n=2$ incoming roads and $m=1$
outgoing road. For every $\rho _{a,0},\rho _{b,0},\rho _{c,0}\in
\left[ 0,\rho _{\max }\right],$ there exists a unique admissible
weak solution $\rho =\left( \rho _{a},\rho _{b},\rho _{c}\right) $
at the junction $J$, satisfying rules (A), (B) and
(C), such that
\[
\rho _{a}\left( 0,\cdot \right) \equiv \rho _{a,0},\text{ }\rho
_{b}\left( 0,\cdot \right) \equiv \rho _{b,0},\text{ }\rho
_{c}\left( 0,\cdot \right) \equiv \rho _{c,0}.
\]
Moreover, there exists a unique $3-$tuple $\left( \widehat{\rho
}_{a}, \widehat{\rho }_{b},\widehat{\rho }_{c}\right) \in \left[
0,\rho _{\max }\right] ^{3}$ such that
\[
\widehat{\rho }_{i}\in \left\{
\begin{tabular}{ll}
$\left\{ \rho _{i,0}\right\} \cup \left] \tau \left( \rho
_{i,0}\right) ,\rho _{\max }
\right] ,$ & if $0\leq \rho _{i,0}\leq \sigma ,$ \\
$\left[ \sigma ,\rho _{\max }\right] ,$ & if $\sigma \leq \rho
_{i,0}\leq \rho _{\max },$
\end{tabular}
\right. i=a,b,
\]
and
\[
\widehat{\rho }_{c}\in \left\{
\begin{tabular}{ll}
$\left[ 0,\sigma \right] ,$ & if $0\leq \rho _{c,0}\leq \sigma ,$ \\
$\left\{ \rho _{c,0}\right\} \cup \left[ 0,\tau \left( \rho
_{c,0}\right) \right[ ,$ & if $\sigma \leq \rho _{c,0}\leq \rho
_{\max },$
\end{tabular}
\right.
\]
and for $i\in \left\{ a,b\right\} $, the solution is given by the
wave $\left( \rho _{i,0},\widehat{\rho }_{i}\right) $, while for
the outgoing road the solution is given by the wave $\left(
\widehat{\rho }_{c},\rho _{c,0}\right) .$
\end{thm}

%
%
%

\subsection{The case of $n$ = 1 incoming road and $m$ = 2 outgoing roads}
\label{sec: 2.2}

Let us now consider the junction with the incoming road $a$
and two outgoing roads $b$ and $c$. 
The distribution matrix $A$, of rule (A), takes the form%
\[
A=\left(
\begin{tabular}{c}
$\alpha $ \\
$1-\alpha $
\end{tabular}
\right) ,
\]
where $\alpha \in \left] 0,1\right[ $ and $\left( 1-\alpha \right)
$ indicate the percentage of cars which, from road $a$, goes to
roads $b$ and $c $, respectively. Thanks to rule (B), the solution
to a RP is:
\begin{equation*}
\widehat{\gamma }=\left( \widehat{\gamma }_{a},\widehat{\gamma
}_{b}, \widehat{\gamma }_{c}\right) =\left( \widehat{\gamma
}_{a},\alpha \widehat{ \gamma }_{a},\left( 1-\alpha \right)
\widehat{\gamma }_{a}\right) ,
\end{equation*}
where
\begin{equation*}
\widehat{\gamma }_{a}=\min \left\{ \gamma _{a}^{\max
},\frac{\gamma _{b}^{\max }}{\alpha },\frac{\gamma _{c}^{\max
}}{1-\alpha }\right\} .
\end{equation*}
Once we have obtained $\widehat{\gamma }_{a},$ $\widehat{\gamma
}_{b}$ and $ \widehat{\gamma }_{c}$, it is possible to find in a
unique way $\widehat{ \rho }_{i}$, $i\in \left\{ a,b,c\right\} $,
reasoning as in the proof of Theorem~\ref{theo:rho-values}.
Then we obtain the following:
\begin{thm}
Consider a junction $J$ with $n=1$ incoming road and $m=2$
outgoing roads. For every $\rho _{a,0},\rho _{b,0},\rho _{c,0}\in
\left[ 0,\rho _{\max }\right] ,$ there exists a unique admissible
weak solution $\rho =\left( \rho _{a},\rho _{b},\rho _{c}\right) $
at the junction $J$, respecting rules (A) and (B),
such that
\[
\rho _{a}\left( 0,\cdot \right) \equiv \rho _{a,0},\text{ }\rho
_{b}\left( 0,\cdot \right) \equiv \rho _{b,0},\text{ }\rho
_{c}\left( 0,\cdot \right) \equiv \rho _{c,0}.
\]
Moreover, there exists a unique $3-$tuple $\left( \widehat{\rho
}_{a}, \widehat{\rho }_{b},\widehat{\rho }_{c}\right) \in \left[
0,\rho _{\max }\right] ^{3}$ such that
\[
\widehat{\rho }_{a}\in \left\{
\begin{tabular}{ll}
$\left\{ \rho _{a,0}\right\} \cup \left] \tau \left( \rho
_{a,0}\right) ,\rho _{\max }
\right] ,$ & if $0\leq \rho _{a,0}\leq \sigma ,$ \\
$\left[ \sigma ,\rho _{\max }\right] ,$ & if $\sigma \leq \rho
_{a,0}\leq \rho _{\max },$
\end{tabular}
\right.
\]
and
\[
\widehat{\rho }_{j}\in \left\{
\begin{tabular}{ll}
$\left[ 0,\sigma \right] ,$ & if $0\leq \rho _{j,0}\leq \sigma ,$ \\
$\left\{ \rho _{j,0}\right\} \cup \left[ 0,\tau \left( \rho
_{j,0}\right) \right[ ,$ & if $\sigma \leq \rho _{j,0}\leq \rho
_{\max },$
\end{tabular}
\right. j=b,c,
\]
and for the incoming road the solution is given by the wave
$\left( \rho _{a,0},\widehat{\rho }_{a}\right) $, while for
$j=b,c,$ the solution is given by the wave $\left( \widehat{\rho
}_{j},\rho _{j,0}\right) .$
\end{thm}

\subsection{The case of $n$ = 2 incoming roads and $m$ = 2 outgoing roads}
\label{sec: 2.3}

Let us now consider the junction with two incoming roads $a$ and $b$
and two outgoing roads $c$ and $d$. 
The distribution matrix $A$, of rule (A), takes the form
\begin{equation}
 A=\left(
\begin{array}{cc}
\alpha  &  \beta\\
1-\alpha & 1-\beta
\end{array} \right),
\label{MatrixA}
\end{equation}
where $\alpha,\beta \in \left] 0,1\right[ $. We assume that 
$\alpha\not=\beta$, otherwise we may have more than one solutions
to the Linear Programming problem, see~\cite{GaravelloPiccoli2006ab}
for details.

First notice that constraints from outgoing roads fluxes can be expressed
as:
\[
\alpha \widehat\gamma_a+\beta\widehat\gamma_b\leq\gamma_c^{max},\quad
(1-\alpha) \widehat\gamma_a+(1-\beta)\widehat\gamma_b\leq\gamma_d^{max}.
\]
Define $P=(\gamma_1,\gamma_2)$ to be the point of intersection of the two lines:
\[
\alpha \gamma_1+\beta \gamma_2=\gamma_c^{max},\quad
(1-\alpha) \gamma_1+(1-\beta) \gamma_2=\gamma_d^{max}.
\]
To express the solution we need to distinguish some cases:\\
{\bf Case a)}. If $\gamma_1\leq \gamma_a^{max}$
and $\gamma_2\leq \gamma_b^{max}$ then the solution is given by:
\[
\widehat\gamma_a=\gamma_1,\quad
\widehat\gamma_b=\gamma_2.
\]
{\bf Case b)}. If $\gamma_1> \gamma_a^{max}$
and $\gamma_2> \gamma_b^{max}$ then the solution is given by:
\[
\widehat\gamma_a=\gamma_a^{max},\quad
\widehat\gamma_b=\gamma_b^{max}.
\]
{\bf Case c)}. Assume $\gamma_1> \gamma_a^{max}$
and $\gamma_2\leq \gamma_b^{max}$. If $\alpha<\beta$
(thus $1-\beta<1-\alpha$)
then the constraint given by outgoing road $c$ is more stringent
than that of outgoing road $d$, thus the solution is given by:
\[
\widehat\gamma_a=\gamma_a^{max},\quad
\widehat\gamma_b=
\mbox{min} (\frac{\gamma_c^{max}-\alpha \gamma_a^{max}}{\beta}, \gamma_b^{max}).
\]
Otherwise, i.e. if $\alpha>\beta$, then the solution is given by:
\[
\widehat\gamma_a=\gamma_a^{max},\quad
\widehat\gamma_b=
\mbox{min} (\frac{\gamma_d^{max}-(1-\alpha) \gamma_a^{max}}{1-\beta}, \gamma_b^{max}).
\]
{\bf Case d)}. Assume $\gamma_1\leq \gamma_a^{max}$
and $\gamma_2> \gamma_b^{max}$. If $\alpha>\beta$
(thus $1-\beta>1-\alpha$)
then the constraint given by outgoing road $c$ is more stringent
than that of outgoing road $d$, thus the solution is given by:
\[
\widehat\gamma_a=
\mbox{min} (\frac{\gamma_c^{max}-\beta \gamma_b^{max}}{\alpha}, \gamma_a^{max}), \quad
\widehat\gamma_b=\gamma_b^{max}.
\]
Otherwise, i.e. if $\alpha<\beta$, then the solution is given by:
\[
\widehat\gamma_a=
\mbox{min} (\frac{\gamma_d^{max}-(1-\beta) \gamma_b^{max}}{1-\alpha}, \gamma_a^{max}), \quad
\widehat\gamma_b=\gamma_b^{max}.
\]

\section{Numerical methods: RKDG}
\label{sec3}
\setcounter{equation}{0}
\setcounter{figure}{0}
\setcounter{table}{0}

Below, we will first describe the RKDG method to discretize the 1D nonlinear traffic flow equations; then we will extend the algorithm to 1D network problems incorporating coupling conditions at the junctions. Finally, we will apply a high order limiter to preserve the upper and lower bounds of the high order solutions.

\subsection{RKDG for 1D hyperbolic equations.}

\noindent
\underline{DG spatial discretization.}
Consider the following spatial discretization: let $I_j=[x_{j-\frac{1}{2}}, x_{j+\frac{1}{2}}]$ for $j=1,\ldots,N_x$ be a partition of $[0,L]$ with $x_{j}=\frac{1}{2}(x_{j-\frac{1}{2}}+x_{j+\frac{1}{2}})$ being the cell center and $\Delta x_j=x_{j+\frac{1}{2}}-x_{j-\frac{1}{2}}$ being the cell size. The semi-discrete DG scheme for the equation \eqref{trafficeq} can be designed as finding numerical solutions $\rho_h$ in a finite dimensional space consisting of piecewise polynomials of degree $k$,
\[
V^k_h = \{u: u|_{I_j} \in P^k, \quad 1\le j \le N_x\}, \quad \mbox{for any non-negative integer $k$},
\]
so that for any test functions $\psi \in V^k_h$,
\begin{equation}
\label{dgeq}
\frac{\partial}{\partial t} \int_{I_j} \rho_h \psi dx = \int_{I_j} f(\rho_h) \partial_x \psi dx -\left( \hat{f}_{j+\frac{1}{2}} \psi^-_{j+\frac{1}{2}} - \hat{f}_{j-\frac{1}{2}} \psi^+_{j-\frac{1}{2}} \right).
\end{equation}
Here and below, the superscripts $\pm$ denotes the right/left limit of the function at a point.
The function $\hat{f}_{j+\frac{1}{2}} = \hat{f} \left(\rho_h(x_{j+\frac{1}{2}}^-), \rho_h(x_{j+\frac{1}{2}}^+)\right)$ is a single valued function at the cell interface, which is defined via an approximate Riemann solver depending on the left and right limits of the DG solutions. One example is the global Lax-Friedrich flux, for which
\beq
\label{eq: flux}
\hat{f} (\rho_h^-, \rho_h^+) = \frac12 (f(\rho_h^-) + f(\rho_h^+)) + \frac{\alpha}{2} (\rho_h^--\rho_h^+), \quad \alpha = \max_{\rho} |f'(\rho)|.
\eeq
For implementation, the approximate solution $\rho_h$ on mesh $I_j$ can be expressed as
\begin{equation}
\rho_h(x,t)=\sum_{l=0}^{k} \rho_j^l(t)\psi_j^l(x),
\end{equation}
where $\{\psi_j^l(x)\}_{l=0}^k$ is the set of basis functions of $P^k(I_j)$. For example, the Legendre polynomials are a local orthogonal basis of $P^k(I_j)$ with
\[
\psi_j^0=1, \quad \psi_j^1=\left(\frac{x-x_j}{\Delta x_j/2}\right), \quad \psi_j^2=\left(\frac{x-x_j}{\Delta x_j/2}\right)^2-\frac{1}{3}, \quad \ldots
\]

\noindent
\underline{RK time discretization in time.}
Eq.~(\ref{dgeq}) is solved in time via the method of lines by a TVD RK method~\cite{shu1988total} in the following form,
\begin{equation}
\begin{split}
&\mathbf{\rho}_h^{(1)}=\mathbf{\rho}_h^n+\Delta t^n L(\mathbf{\rho}_h^n), \\
&\mathbf{\rho}_h^{(2)}=\frac{3}{4}\mathbf{\rho}_h^n+\frac{1}{4}\mathbf{\rho}_h^{(1)}+\frac{1}{4}\Delta t^n L(\mathbf{\rho}_h^{(1)}), \\
&\mathbf{\rho}_h^{n+1}=\frac{1}{3}\mathbf{\rho}_h^n+\frac{2}{3}\mathbf{\rho}_h^{(2)}+\frac{2}{3}\Delta t^n L(\mathbf{\rho}_h^{(2)}),
\end{split}
\label{rk3}
\end{equation}
where $L$ is the spatial operator, which denotes the R.H.S. of eq. \eqref{dgeq}, and $\Delta t^n$ is the numerical time step.

\noindent
\underline{TVB limiters.}
When the solutions contain shocks, the TVB limiters proposed by Cockburn and Shu~\cite{cockburn1989tvb} will be used to eliminate spurious oscillations and enforce stability.
Let $\bar{\rho}_j$ be the cell averages of the numerical solution on cell $I_j$, and let
\begin{equation}
\tilde{\rho}_j \doteq \rho_h(x_{j+1/2}^-)-\bar{\rho}_j,\quad \tilde{\tilde{\rho}}_j\doteq -(\rho(x_{j-1/2}^+)-\bar{\rho}_j).
\end{equation}
The TVB limiter is used to adjust $\tilde{\rho}_j, \tilde{\tilde{\rho}}_j$ as follows,
\begin{equation}
\tilde{\rho}_j^{(\mbox{mod})}=\bar{m}(\tilde{\rho}_j,\Delta_+ \bar{\rho}_j,\Delta_- \bar{\rho}_j),\quad \tilde{\tilde{\rho}}_j^{(\mbox{mod})}=\bar{m}(\tilde{\tilde{\rho}}_j,\Delta_+ \bar{\rho}_j,\Delta_- \bar{\rho}_j),
\end{equation}
where $\Delta_+ \bar{\rho}_j=\bar{\rho}_{j+1}-\bar{\rho}_j$, and $\Delta_- \bar{\rho}_j=\bar{\rho}_j-\bar{\rho}_{j-1}$, and the modified minmod function $\bar{m}$ is defined by
\begin{equation}
\bar{m}(a_1,a_2,a_3)=\begin{cases}
a_1, \quad &\mbox{if}\, |a_1| \le M \Delta x_j^2 \\
s\, \mbox{min}(|a_1|,|a_2|,|a_3|), \quad &\mbox{if}\,|a_1|> M \Delta x_j^2, \,\mbox{sign}(a_1)=\mbox{sign}(a_2)=\mbox{sign}(a_3)=s,\\
0, \quad & \mbox{otherwise}.
\end{cases}
\end{equation}
The limited $\tilde{\rho}_j^{(\mbox{mod})}$ and $\tilde{\tilde{\rho}}_j^{(\mbox{mod})}$ are then used to recover the new point values,
\begin{equation}
\rho_{j+1/2}^{(\mbox{mod}),-}=\bar{\rho}_j+\tilde{\rho}_j^{(\mbox{mod})},\quad \rho_{j-1/2}^{(\mbox{mod}),+}=\bar{\rho}_j-\tilde{\tilde{\rho}}_j^{(\mbox{mod})}.
\end{equation}
With the modified $\rho_{j+1/2}^{(\mbox{mod}),-}$, $\rho_{j-1/2}^{(\mbox{mod}),+}$ as numerical solutions at the cell boundaries, as well as the cell average $\bar{\rho}_j$, we can construct a unique $P^2$ polynomial as the modified numerical solution.

\noindent
\underline{Boundary conditions.} Depending on the flow directions at the boundaries, one can prescribe either the inflow or outflow boundary conditions for the open boundaries. The RKDG method is well-known for its compactness, for which the inflow information or outflow information from extrapolation can be directly used in evaluating the flux specified in equation \eqref{eq: flux} at the boundary.
Treatment of the coupling (boundary) conditions at network's vertices is described next.

\subsection{RKDG for hyperbolic networks.}
 \label{junctionflux}
The coupling conditions within a network, consisting of many incoming and outgoing roads with different junction points, are described in a general setting in Section~\ref{sec2}. Below, we consider specific ways of constructing numerical coupling conditions, i.e.,
the numerical fluxes, at a single junction point of the following types: (a) one incoming and one outgoing road, (b) two incoming roads and one outgoing road, (c) one incoming and two outgoing roads, (d) two incoming and two outgoing roads.
Similar coupling conditions can be derived for more complicated cases based on the rules (A), (B), (C) specified in Section~\ref{sec2}.

Below, we
assume that $\gamma_m^{max}$ is chosen following equation \eqref{fluxUno} with $\rho$ being the right limit of the numerical solution on the $N_m^{th}$ cell of an incoming road $m$ at the junction point; $\gamma_n^{max}$ is chosen following equation \eqref{fluxDue}
with $\rho$ being the left limit of the numerical solution on the first cell of an outgoing road $n$ at the junction point.

%

\noindent
\underline{One incoming road and one outgoing road.}
Let us consider the junction with one incoming road $a$ and one outgoing road $b$. The coupling conditions are
\begin{equation}
\widehat{\gamma}_a=\widehat{\gamma}_b=\widehat{\gamma}, \quad \widehat{\gamma}=\mbox{min}\, \{\gamma_a^{max}, \gamma_b^{max}\}.
\label{oneone1}
\end{equation}
The numerical fluxes of the junction point at incoming and outgoing roads are set to be
\begin{equation}
 \hat{f}_{N_a+\frac{1}{2}}^a=\widehat{\gamma}_a, \quad \hat{f}_{\frac{1}{2}}^b=\widehat{\gamma}_b.
 \label{oneone2}
\end{equation}

\noindent
\underline{One incoming road and two outgoing roads}. Let us consider the junction with one incoming road $a$ and two outgoing roads $b$ and $c$. With the notation introduced for $\alpha$ in Section~\ref{sec: 2.2},
the coupling conditions are
\begin{equation}
\widehat{\gamma}_a=\mbox{min} \left\{\gamma_a^{max}, \frac{\gamma_b^{max}}{\alpha}, \frac{\gamma_c^{max}}{1-\alpha}\right\}, \quad \widehat{\gamma_b}=\alpha \widehat{\gamma}_a,
\quad \widehat{\gamma_c}=(1-\alpha) \widehat{\gamma}_a.
\end{equation}
The numerical fluxes at the junction point at the incoming and outgoing roads are
\begin{equation}
 \hat{f}_{N_a+\frac{1}{2}}^a=\widehat{\gamma}_a, \quad \hat{f}_{\frac{1}{2}}^b=\widehat{\gamma}_b, \quad \hat{f}_{\frac{1}{2}}^c=\widehat{\gamma}_c.
\end{equation}

\noindent
\underline{Two incoming roads and one outgoing road}. Let us consider the junction with two incoming roads $a, b$, and one outgoing road $c$. With the notation introduced for $q$ in Section~\ref{sec: 2.1}, the coupling conditions are the following.
For the case of $\gamma_a^{max}+\gamma_b^{max} < \gamma_c^{max}$, we have
\begin{equation}
 \quad \widehat{\gamma}_a=\gamma_a^{max}, \quad \widehat{\gamma}_b=\gamma_b^{max}, \quad \widehat{\gamma}_c=\widehat{\gamma}_a+\widehat{\gamma}_b.
 \label{twooneflux1}
\end{equation}
For the case of $\gamma_a^{max}+\gamma_b^{max} > \gamma_c^{max}$, the coupling conditions are
\begin{eqnarray}
\widehat{\gamma}_a=q \gamma_c^{max}, \ \widehat{\gamma}_b=(1-q) \gamma_c^{max}, \ \widehat{\gamma}_c=\gamma_c^{max}, \quad &\mbox{if} \ \gamma_a^{max}\ge q \gamma_c^{max}, \ \gamma_b^{max}\ge(1-q) \gamma_c^{max},\\ \notag
\widehat{\gamma}_a=\gamma_a^{max}, \ \widehat{\gamma}_b=\gamma_c^{max}-\gamma_a^{max}, \ \widehat{\gamma}_c=\gamma_c^{max}, \quad &\mbox{if} \ \gamma_a^{max}<q \gamma_c^{max}, \ \gamma_b^{max}\ge(1-q) \gamma_c^{max},\\ \notag
\widehat{\gamma}_a=\gamma_c^{max}-\gamma_b^{max},\ \widehat{\gamma}_b=\gamma_b^{max}, \  \widehat{\gamma}_c=\gamma_c^{max}, \quad &\mbox{if} \ \gamma_a^{max}\ge q \gamma_c^{max}, \ \gamma_b^{max}<(1-q) \gamma_c^{max}.\notag
 \label{twooneflux2}
\end{eqnarray}
The numerical fluxes at the junction point between the incoming and outgoing roads are
\begin{equation}
 \hat{f}_{N_a+\frac{1}{2}}^a=\widehat{\gamma}_a, \quad \hat{f}_{N_b+\frac{1}{2}}^b=\widehat{\gamma}_b, \quad \hat{f}_{\frac{1}{2}}^c=\widehat{\gamma}_c.
  \label{twooneflux3}
\end{equation}

\noindent
\underline{Two incoming roads and two outgoing roads}. Let us consider the junction with two incoming roads $a, b$, and two outgoing roads $c, d$. The coupling conditions follow those described in Section~\ref{sec: 2.3} for $\widehat{\gamma}_a, \widehat{\gamma}_b$, with $(\widehat{\gamma}_c, \widehat{\gamma}_d)^T=A (\widehat{\gamma}_a, \widehat{\gamma}_b)^T$.
The numerical fluxes at the junction point on incoming and outgoing roads are
\begin{equation}
 \hat{f}_{N_a+\frac{1}{2}}^a=\widehat{\gamma}_a, \quad \hat{f}_{N_b+\frac{1}{2}}^b=\widehat{\gamma}_b, \quad \hat{f}_{\frac{1}{2}}^c=\widehat{\gamma}_c,\quad \hat{f}_{\frac{1}{2}}^d=\widehat{\gamma}_d.
\label{twotwoflux}
 \end{equation}
\begin{rem} One distinct property of the DG method compared with the other high order methods, e.g. finite volume or finite difference methods, for solving hyperbolic equations is the compactness of the scheme. Specifically, the high order fluxes depend only on the direct neighboring cells. Such property shows great advantage in prescribing boundary conditions at the junction points of hyperbolic networks.
\end{rem}

\subsection{Bound preserving numerical solutions}

In the traffic flow model, it is known that $\rho(x, t) \in [0, \rho_{max}]$. However, such a property does not hold for high order numerical solutions in general.
To preserve the theoretical bounds on the RKDG solution,
in this subsection we propose to apply the limiter proposed in~\cite{zhang2010maximum}.
The application of this limiter is based on the fact that a first order monotone scheme with piecewise constant numerical solutions for  network problems satisfies the following property.

Let $\bar\rho^n_j$ denote numerical approximations to solutions at cell $I_j$ at time $t^n$, and let $\Delta x$ be the spatial mesh size for a uniform mesh. (Similar results hold for nonuniform meshes.)

\begin{defn} (Monotone scheme)
A first order monotone scheme for the 1-D hyperbolic equation~\eqref{trafficeq} can be written in the following form,
\[
\bar\rho_{j}^{n+1} = \bar\rho_{j}^{n} - \frac{\Delta t}{\Delta x} (\hat{f}_{j+\frac12} -\hat{f}_{j-\frac12}).
\]
Here $\hat{f}_{j+\frac12} = \hat{f}(\bar\rho_{j}^n, \bar\rho_{j+1}^n)$ is a monotone flux, that is, it is a non-decreasing function with respect to the first argument, and non-increasing function with respect to the second argument. Similarly for $\hat{f}_{j-\frac12}$.
\end{defn}
It can be easily shown that in the monotone scheme, $\bar\rho_{j}^{n+1}\doteq G(\bar\rho_{j-1}^n, \bar\rho_{j}^n, \bar\rho_{j+1}^n)$ is a non-decreasing function with respect to $\bar\rho_{j-1}^{n}$, $\bar\rho_{j}^{n}$ and $\bar\rho_{j+1}^{n}$, if a proper CFL condition is satisfied. Hence,
\beq
\label{eq: mmp}
0= G(0, 0, 0) \le \bar\rho_{j}^{n+1} \le G(\rho_{max}, \rho_{max}, \rho_{max}) = \rho_{max}, \ \forall n,j,
\eeq
provided that the initial condition $\bar\rho_{j}^{0} \in [0, \rho_{max}]$, leading to the bound preserving property of the numerical solution.
 The proposition below is a generalization of such a result to hyperbolic network problems. Without loss of generality, we consider the Godunov flux as our numerical flux, with
 \beq
 \hat{f} (\bar\rho_j, \bar\rho_{j+1}) = \left \{
 \begin{array} {ll}
 \min_{\rho\in [\bar\rho_j, \bar\rho_{j+1}]} f(\rho), & \mbox{if} \   \bar\rho_j \le \bar\rho_{j+1}\\
 \max_{\rho\in [\bar\rho_{j+1}, \bar\rho_{j}]} f(\rho), & \mbox{otherwise.}
 \end{array}
 \right.
 \eeq
\begin{prop}
Consider a first order monotone scheme with Godnov flux as a numerical flux for the 1-D hyperbolic equation \eqref{trafficeq} holding on each road
in a network, satisfying the coupling conditions at the junctions respecting the rules (A), (B), and ( C),  with equations \eqref{fluxUno} and \eqref{fluxDue} specified in Section~\ref{sec2}. Then, the numerical solution satisfies the bounds \eqref{eq: mmp}.
\end{prop}
\begin{proof}
It is sufficient to prove the statement for the boundary elements adjacent to junctions. We consider the left-most element on a road ($j=1$), which is an outgoing road at a junction.
From equation \eqref{fluxDue}, together with the rules (A), (B), and (C), we have $\hat{f}_\frac12 \ge 0 = \hat{f}(0, \bar\rho^n_1)$. Hence
\[
\bar\rho_{1}^{n+1} = \bar\rho_{1}^{n} - \frac{\Delta t}{\Delta x} (\hat{f}_{\frac32} -\hat{f}_{\frac12}) \ge \bar\rho_{1}^{n} - \frac{\Delta t}{\Delta x} (\hat{f}_{\frac32} -\hat{f}(0, \bar\rho^n_1)) \ge 0,
\]
where the last inequality is due to the monotonicity of the scheme.
In order to prove $\bar\rho_{1}^{n+1} \le \rho_{max}$, we discuss two cases:
\begin{enumerate}
\item[(a)] When $\bar\rho^n_1 \le \sigma$, we have $\hat{f}_\frac12 \le f(\sigma) = \hat{f}(\sigma, \bar\rho^n_1)$, hence
\[
\bar\rho_{1}^{n+1} = \bar\rho_{1}^{n} - \frac{\Delta t}{\Delta x} (\hat{f}_{\frac32} -\hat{f}_{\frac12}) \le \bar\rho_{1}^{n} - \frac{\Delta t}{\Delta x} (\hat{f}_{\frac32} -\hat{f}(\sigma, \bar\rho^n_1)) \le \rho_{max},
\]
where the last inequality is due to the monotonicity of the scheme.
\item[(b)] Similarly, when $\bar\rho^n_j > \sigma$, we have $\hat{f}_\frac12 \le f(\bar\rho^n_1) = \hat{f}(\bar\rho^n_1, \bar\rho^n_1)$, hence
\[
\bar\rho_{1}^{n+1} = \bar\rho_{1}^{n} - \frac{\Delta t}{\Delta x} (\hat{f}_{\frac32} -\hat{f}_{\frac12}) \le \bar\rho_{1}^{n} - \frac{\Delta t}{\Delta x} (\hat{f}_{\frac32} -\hat{f}(\bar\rho^n_1, \bar\rho^n_1)) \le \rho_{max}.
\]
\end{enumerate}
Similar procedures can be done to prove the property for the right-most element on a road.
\end{proof}

In~\cite{zhang2010maximum}, a maximum principle preserving limiter is introduced for the RKDG scheme to preserve the maximum principle of the numerical solutions for hyperbolic PDEs, with the assumption that a first order monotone scheme satisfies the same property. The procedure of the maximum principle preserving limiter can be viewed as controlling the maximum and minimum of the numerical solution (polynomials on discretized cells) by a linear rescaling around cell averages.
Such a procedure can be applied to control the bounds of the high order RKDG solutions for hyperbolic network problems.
In particular, we would like to modify the numerical solution $\rho_h(x)$ to ${\rho}^*_h(x)$, approximating a function $\rho(x)$ on a cell $I_j$, such that it satisfies
\bit
\item Accuracy: for smooth function $\rho(x)$, $\|\rho_h(x) - {\rho}^*_h(x)\| = \mathcal{O}(\Delta x^{k+1})$, on $I_j$;
\item Mass conservation property: $\int_{I_j} {\rho}^*_h (x) dx =  \int_{I_j} \rho_h(x)dx \doteq \bar{\rho}_j$;
\item Bounds-preserving: ${\rho}^*_h(x) \in [0, \rho_{max}]$ on $I_j$.
\eit
In order to achieve the above mentioned properties, one can apply the following limiter
\beq
\label{eq: MPP}
{\rho}^*_h(x) = \theta(\rho_h(x)-\bar{\rho}_j) + \bar{\rho}_j, \quad \theta = \min
\left\{ \bigg |\frac{\rho_{max}-\bar{\rho}_j}{M_j-\bar{\rho}_j} \bigg |, \bigg |\frac{\bar{\rho}_j}{m_j-\bar{\rho}_j} \bigg|, 1
\right\},
\eeq
where $M_j$ and $m_j$ are the maximum and the minimum of $\rho_h(x)$ at Legendre Gauss-Lobatto quadrature points for the cell $I_j$.
It can be easily checked that with the application of such a limiter, the conservation and bound preserving properties of the numerical solution are satisfied. Furthermore, it
was proved~\cite{zhang2010maximum} that such a limiting process maintains the original $(k+1)^{th}$ order accuracy of the approximation.

Since the first order monotone scheme preserves the bounds for hyperbolic network problems, following similar procedures as in~\cite{zhang2010maximum} one can show that
the cell averages of the high order scheme are also well bounded, i.e. $\bar{\rho}_j \in [0, \rho_{max}]$, $\forall j$,
under the additional CFL constraint:
\[
\max_{\rho} |f'(\rho)| \frac{\Delta t}{\Delta x} \le \min_i w_i,
\]
where  $w_i$'s are the quadrature weights in the Legendre Gauss-Lobatto quadrature rule on a standard interval $[-\frac12, \frac12]$. Hence, the above limiter can be applied to the proposed RKDG scheme for hyperbolic networks. We also remark that if $\rho_{max}$ varies among different roads within a network, one can apply the similar limiter with the appropriate upper bounds on density.

\section{Numerical examples}
\label{sec4}
\setcounter{equation}{0}
\setcounter{figure}{0}
\setcounter{table}{0}

In this section, we reproduce simulation results from~\cite{bretti2006numerical}, using the high order RKDG method, discussed above,
to compare the performance of the proposed high order scheme with the first order scheme used in~\cite{bretti2006numerical}.
We also present several new examples with more complicated solution structures to showcase the advantages of high order schemes.
In our numerical examples, for the third-order TVD Runge-Kutta method~(\ref{rk3}), we take CFL=1.0, 0.33, 0.20, 0.14 for $P^0$, $P^1$, $P^2$ and $P^3$ solution spaces corresponding to DG schemes with first to fourth spatial orders respectively.
The time step $\Delta t^n=\mbox{CFL} \Delta x$ for $P^0$, $P^1$ and $P^2$ solution spaces, while $\Delta t^n=\mbox{CFL} \Delta x^{\frac{4}{3}}$ for the $P^3$ solution space.
The cell size is $1/40$ and the reference solutions $P_{ref}^0$ are obtained by first order RKDG (finite volume method) with cell size $1/1600$ in all examples, except the accuracy test.


\subsection{Accuracy test}
The first test is to solve the traffic flow equation~(\ref{trafficeq}) with the following flux function
\beq
\label{eq: flux_function}
f(\rho)=\rho(1-\rho),\quad \rho \in [0,1],
\eeq
with the initial condition
\begin{equation}
 \rho(x,0)=0.5+0.5 \sin(2 \pi x).
\label{accu_init}
\end{equation}
The computational domain is  $[0,1]$ with periodic boundary condition. We compute the solutions up to time $t=0.1$.
Newton's method is used to get the reference solution. We use a smaller CFL number, i.e. CFL=0.05 for the $P^2$ and $P^3$ cases to ensure that the spatial error dominates, so that the spatial order of accuracy can be observed for the scheme with the BP limiter.
The results without and with BP limiter are shown in Tables~\ref{table1} and~\ref{table2}.
One can observe the $(k+1)^{st}$-order convergence rate  for $P^k (k=0,1,2,3)$ solution spaces for the scheme with or without the BP limiter.
The results without BP limiter show that the regular RKDG scheme produces numerical solutions that overshoot and undershoot the bounds of the exact solution. With the BP limiter, one can see that the scheme produces results that respect the bounds of the physical solutions.

\begin{table}
\centering
 \caption{Accuracy test, $L^1$ and $L^\infty$ errors and orders, minimum and maximum of numerical solutions for the initial condition~(\ref{accu_init})
 without BP limiter, for $P^0$, $P^1$, $P^2$ and $P^3$ solution spaces. }
  \begin{tabular}{|c|c|c|c|c|c|c|c|}
    \hline
    &N  &  $L^1$ error & order   & $L^\infty$ error  & order & min &max \\\hline
\multirow{7}{*}{$P^0$}
  &10  &0.28E-01  &--    &0.30E+00  &-- &0.000000 &1.000000\\ \cline{2-8}
  &20  &0.14E-01  &0.95  &0.21E+00  &0.48&0.000000 &1.000000\\ \cline{2-8}
  &40  &0.73E-02  &0.97  &0.12E+00  &0.84&0.000000 &1.000000\\ \cline{2-8}
  &80  &0.37E-02  &0.98  &0.66E-01  &0.86&0.000000 &1.000000\\ \cline{2-8}
  &160  &0.19E-02  &0.99  &0.34E-01  &0.94&0.000000 &1.000000\\ \cline{2-8}
  &320  &0.93E-03  &1.00  &0.17E-01  &0.98&0.000000 &1.000000\\ \hline
\multirow{7}{*}{$P^1$}
  &10  &0.46E-02  &--    &0.87E-01  &--&-0.056360 &1.056360\\ \cline{2-8}
  &20  &0.11E-02  &2.12  &0.29E-01  &1.58&-0.009971 &1.009971\\ \cline{2-8}
  &40  &0.25E-03  &2.09  &0.72E-02  &2.02&-0.002324 &1.002324\\ \cline{2-8}
  &80  &0.61E-04  &2.03  &0.19E-02  &1.91&-0.000549 &1.000549\\ \cline{2-8}
  &160  &0.15E-04  &2.01  &0.49E-03  &1.96&-0.000133 &1.000133\\ \cline{2-8}
  &320  &0.38E-05  &2.00  &0.13E-03  &1.98&-0.000033 &1.000033\\ \hline
\multirow{7}{*}{$P^2$}
  &10  &0.28E-03  & --   &0.69E-02  & -- &0.000000 &1.000000 \\  \cline{2-8}
  &20  &0.48E-04  &2.55  &0.18E-02  &1.97&0.000000 &1.000000 \\ \cline{2-8}
  &40  &0.85E-05  &2.49  &0.71E-03  &1.32&0.000000 &1.000000 \\ \cline{2-8}
  &80  &0.12E-05  &2.83  &0.11E-03  &2.73 &0.000000 &1.000000\\ \cline{2-8}
  &160 &0.16E-06  &2.89  &0.21E-04  &2.32 &0.000000 &1.000000 \\  \cline{2-8}
  &320  &0.22E-07  &2.88  &0.42E-05  &2.35&0.000000 &1.000000\\ \hline
\multirow{7}{*}{$P^3$}
  &10  &0.44E-04  & --   &0.24E-02  & -- &0.000000 &1.000000 \\  \cline{2-8}
  &20  &0.52E-05  &3.11  &0.84E-03  &1.55 &-0.000043 &1.000043\\ \cline{2-8}
  &40  &0.24E-06  &4.45  &0.72E-04  &3.54 &-0.000002 &1.000002\\ \cline{2-8}
  &80  &0.13E-07  &4.21  &0.49E-05  &3.89 &0.000000 &1.000000\\ \cline{2-8}
  &160 &0.77E-09  &4.06  &0.32E-06  &3.92 &0.000000 &1.000000 \\  \cline{2-8}
  &320  &0.47E-10  &4.01  &0.20E-07  &3.98&0.000000 &1.000000\\ \hline
  \end{tabular}
\label{table1}
\end{table}

\begin{table}
\centering
 \caption{Accuracy test, $L^1$ and $L^\infty$ errors and orders, minimum and maximum of numerical solutions for the initial condition~(\ref{accu_init})
 with BP limiter, for $P^0$, $P^1$, $P^2$ and $P^3$ solution spaces. }
  \begin{tabular}{|c|c|c|c|c|c|c|c|}
    \hline
    &N  &  $L^1$ error & order   & $L^\infty$ error  & order & min &max \\\hline
\multirow{7}{*}{$P^0$}
  &10  &0.28E-01  &--    &0.30E+00  &--&0.000000 &1.000000 \\ \cline{2-8}
  &20  &0.14E-01  &0.95  &0.21E+00  &0.48&0.000000 &1.000000\\ \cline{2-8}
  &40  &0.73E-02  &0.97  &0.12E+00  &0.84&0.000000 &1.000000\\ \cline{2-8}
  &80  &0.37E-02  &0.98  &0.66E-01  &0.86&0.000000 &1.000000\\ \cline{2-8}
  &160  &0.19E-02  &0.99  &0.34E-01  &0.94&0.000000 &1.000000\\ \cline{2-8}
  &320  &0.93E-03  &1.00  &0.17E-01  &0.98&0.000000 &1.000000\\ \hline
\multirow{7}{*}{$P^1$}
  &10  &0.59E-02  &--    &0.95E-01  &--&0.000000 &1.000000\\ \cline{2-8}
  &20  &0.11E-02  &2.37  &0.30E-01  &1.64&0.000000 &1.000000\\ \cline{2-8}
  &40  &0.26E-03  &2.14  &0.73E-02  &2.07&0.000000 &1.000000\\ \cline{2-8}
  &80  &0.62E-04  &2.04  &0.19E-02  &1.92&0.000000 &1.000000\\ \cline{2-8}
  &160  &0.15E-04  &2.02  &0.49E-03  &1.97&0.000000 &1.000000\\ \cline{2-8}
  &320  &0.38E-05  &2.01  &0.13E-03  &1.98&0.000000 &1.000000\\ \hline
\multirow{7}{*}{$P^2$}
  &10  &0.29E-03  & --   &0.54E-02  & -- &0.000000 &1.000000 \\  \cline{2-8}
  &20  &0.48E-04  &2.58  &0.17E-02  &1.67&0.000000 &1.000000 \\ \cline{2-8}
  &40  &0.85E-05  &2.49  &0.71E-03  &1.27 &0.000000 &1.000000\\ \cline{2-8}
  &80  &0.12E-05  &2.82  &0.11E-03  &2.73 &0.000000 &1.000000\\ \cline{2-8}
  &160 &0.16E-06  &2.88  &0.21E-04  &2.32 &0.000000 &1.000000 \\  \cline{2-8}
  &320  &0.22E-07  &2.87  &0.42E-05  &2.35&0.000000 &1.000000\\ \hline
\multirow{7}{*}{$P^3$}
  &10  &0.44E-04  & --   &0.24E-02  & -- &0.000000 &1.000000 \\  \cline{2-8}
  &20  &0.61E-05  &2.86  &0.84E-03  &1.54&0.000000 &1.000000 \\ \cline{2-8}
  &40  &0.26E-06  &4.56  &0.72E-04  &3.54&0.000000 &1.000000 \\ \cline{2-8}
  &80  &0.13E-07  &4.31  &0.49E-05  &3.89&0.000000 &1.000000 \\ \cline{2-8}
  &160 &0.79E-09  &4.05  &0.32E-06  &3.92 &0.000000 &1.000000 \\  \cline{2-8}
  &320  &0.50E-10  &3.98  &0.20E-07  &3.98&0.000000 &1.000000\\ \hline
  \end{tabular}
\label{table2}
\end{table}

We also tested the scheme with the following initial condition,
\begin{eqnarray}
\rho(0,x)=\left\{
\begin{array}{ll}
1, &\quad \mbox{if} \ x\in [0, 0.3] \cup [0.6, 1], \\
0, &\quad \mbox{otherwise}.
\end{array}\right.
\label{accu_init2}
\end{eqnarray}
Figure~\ref{fig_accu_dis} shows the RKDG solutions with $P^1$ solution space with and without BP limiter, superimposed over the exact solution.
The numerical solution obtained using the scheme without the BP limiter displays some oscillations with the solution outside the physical bounds, while the results obtained using the scheme with the BP limiter fall  inside the bounds, and approximate well the exact solution.
\begin{figure}[htb]
\begin{center}
\includegraphics[width=3.in]{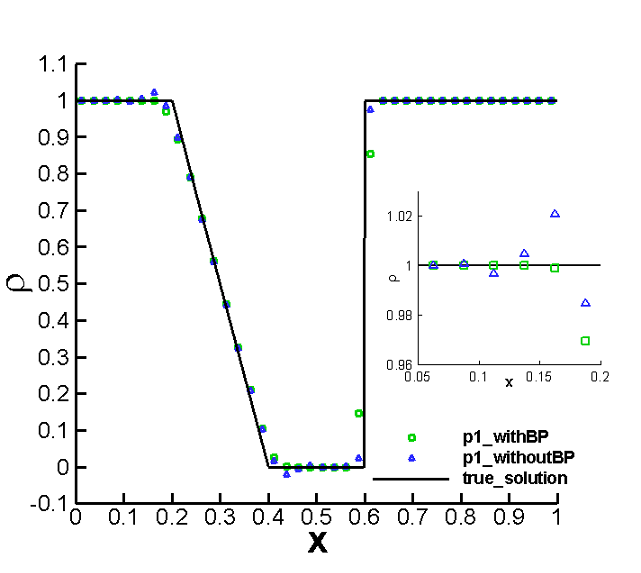}
\end{center}
\caption{Accuracy test with initial data~(\ref{accu_init2}). The numerical solution is produced by RKDG with $P^1$ polynomial space with cell size $1/40$.}
\label{fig_accu_dis}
\end{figure}

\subsection{Bottleneck}
The simplest traffic flow model on networks is represented by the bottleneck problem.
The conservation of cars is always expressed by~\eqref{trafficeq}, supplemented with initial and boundary conditions.
The bottleneck problem models  a road with different widths, hence different flux functions along different parts of the road. Denote the separation point between the two parts of the road by $S$.
We consider a road parametrized on $[0, 2]$ with $S=1$.
We may consider the road as composed of two different roads. Let $\rho_l$ be the traffic density to the left of $S$ on $[0, 1]$
(wider part) and $\rho_r$ be the traffic density to the right of $S$ on $[1, 2]$ (narrower part).
The wider part can be viewed as incoming road and the narrower part can be viewed as outgoing road.
The flux function $f_1(\rho)$ in the wider part is given by eq.~(\ref{eq: flux_function}),
while the flux function in the narrower part is given by
\begin{equation}
 f_2(\rho)=\rho(1-\frac{3}{2}\rho), \ \rho \in [0,2/3]. \notag
\end{equation}
The maximum of the fluxes is unique:
\begin{equation}
 f_1(\sigma_1)=\underset {[0,1]}{\max}\ f_1(\rho)=\frac{1}{4}, \ \mbox{with}\ \sigma_1=\frac{1}{2}; \quad
 f_2(\sigma_2)=\underset {[0,2/3]}{\max}\ f_2(\rho)=\frac{1}{6}, \ \mbox{with} \ \sigma_2=\frac{1}{3}.\notag
\end{equation}

We first consider the following initial and boundary data:
\begin{equation}
 \rho_1(t=0,x)=0.66, \ x\in [0, 1];  \quad \rho_2(t=0,x)=0.66,\ x\in [1, 2]; \quad
 \rho_{1,b}(t, x=0)=0.25.
 \label{testb1}
\end{equation}
The initial value $0.66$ is very close to the maximum value that can be absorbed by road 2, after a short time, e.g. at $T=0.5$, the formation of a traffic jam can be observed, see Figure~\ref{figtestb1}.

We then consider the following initial and boundary data:
\begin{equation}
 \rho_1(t=0,x)= 0, \ x\in[0, 1];  \quad  \rho_2(t=0,x)=0, \ x\in [1, 2]; \quad
 \rho_{1,b}(t,x=0)=0.4.
\label{testb2}
\end{equation}
Since $\rho_{1,b}>\bar{\rho} \simeq 0.21$, there is a jam formation as showed by Figure~\ref{figtestb2}.
The results obtained here are comparable with those in~\cite{bretti2006numerical}.

Finally, we consider the following initial and boundary data:
\begin{equation}
 \rho_1(t=0,x)=0.4+0.2 \sin(5 \pi x),  \ x\in [0, 1]; \quad \rho_2(t=0,x)=0.66, \ x\in [1, 2]; \quad
 \rho_{1,b}(t,x=0)=0.25.
\label{testb3}
\end{equation}
The numerical results are presented in Figure~\ref{figtestb3}. In the presented numerical results, DG solutions with $P^1$ and $P^2$ polynomial spaces have better performance and less numerical diffusion with relatively coarse mesh size compared with the first order scheme.

\begin{figure}[htb]
\begin{center}
\includegraphics[width=2.in]{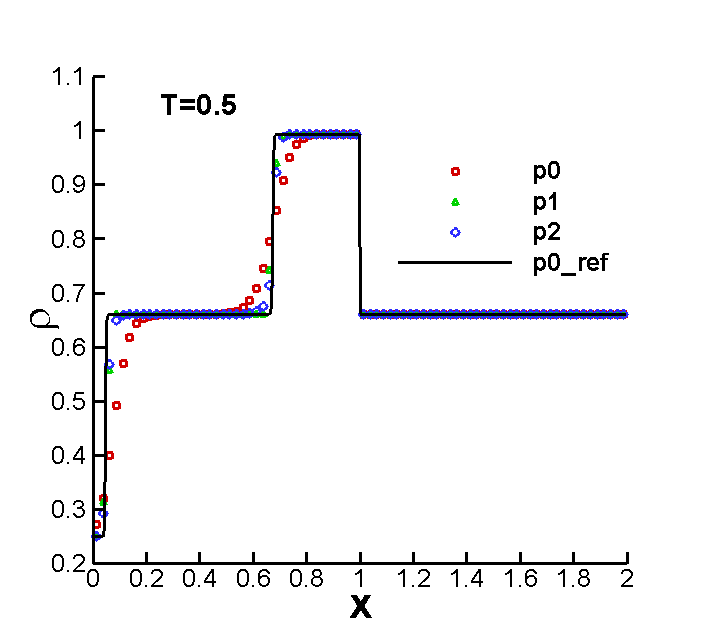}
\includegraphics[width=2.in]{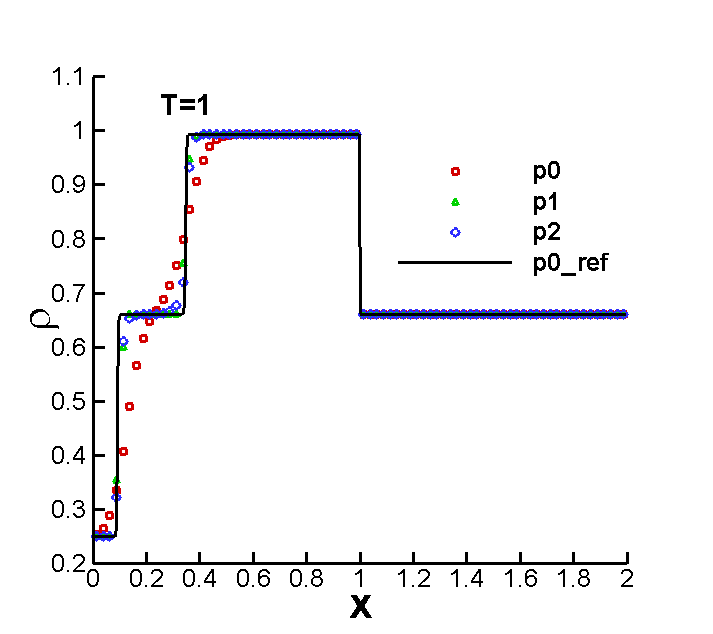}
\includegraphics[width=2.in]{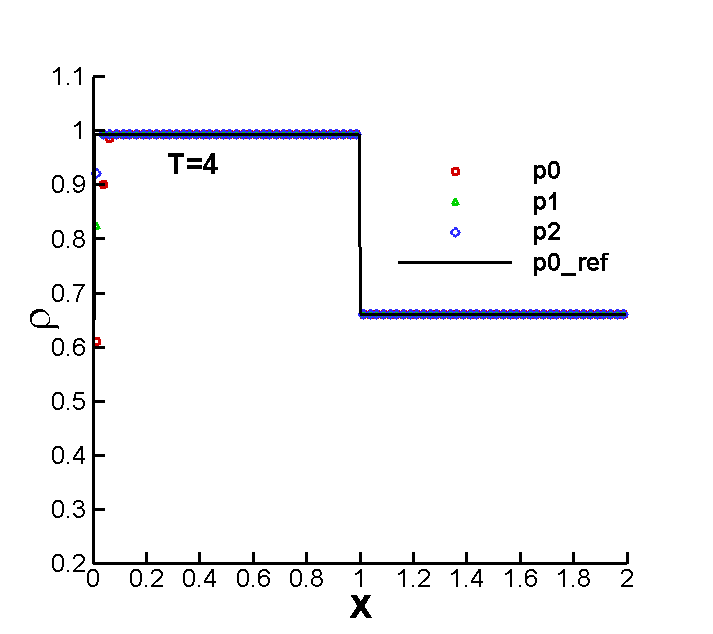}
\end{center}
\caption{Bottleneck problem with initial and boundary data~(\ref{testb1}).
}
\label{figtestb1}
\end{figure}

\begin{figure}[htb]
\begin{center}
\includegraphics[width=2.in]{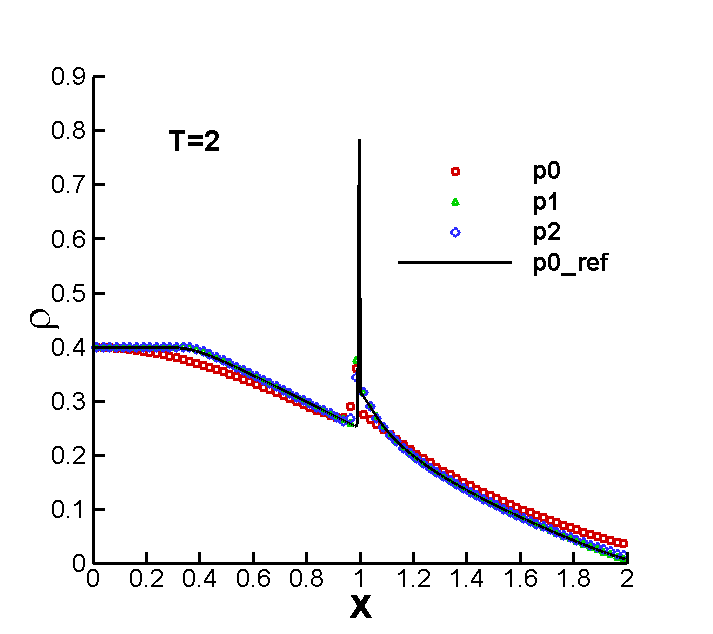}
\includegraphics[width=2.in]{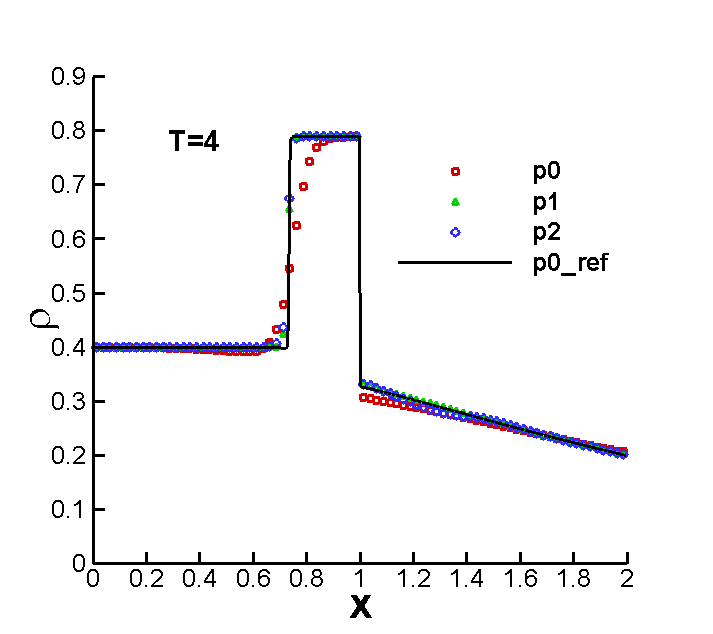}
\includegraphics[width=2.in]{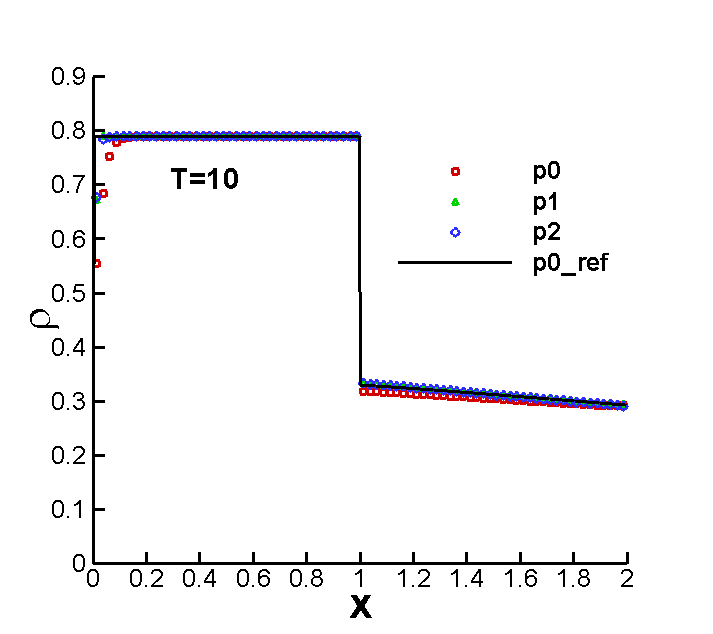}
\end{center}
\caption{
Bottleneck problem with initial and boundary data~(\ref{testb2}). 
}
\label{figtestb2}
\end{figure}

\begin{figure}[htb]
\begin{center}
\includegraphics[width=2.in]{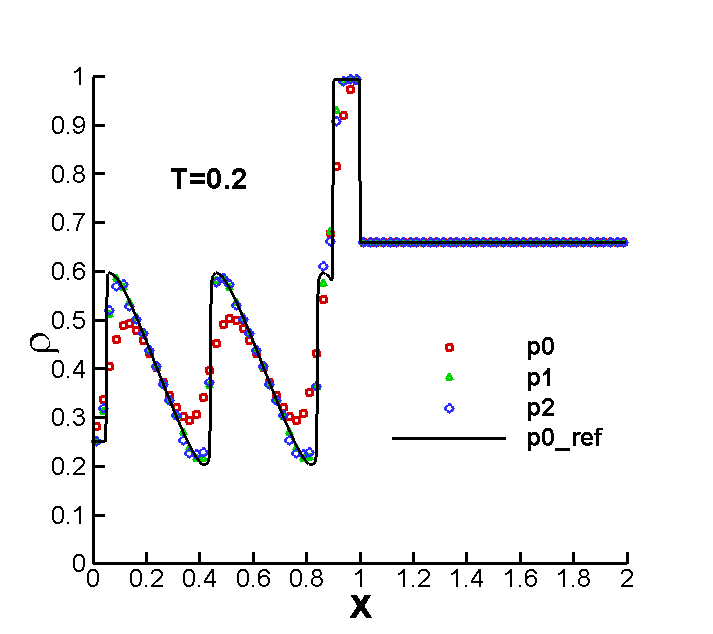}
\includegraphics[width=2.in]{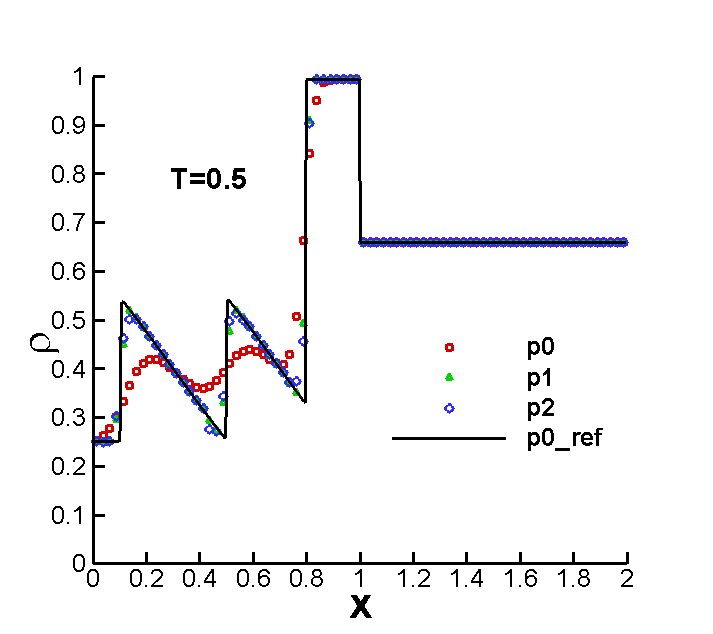}
\includegraphics[width=2.in]{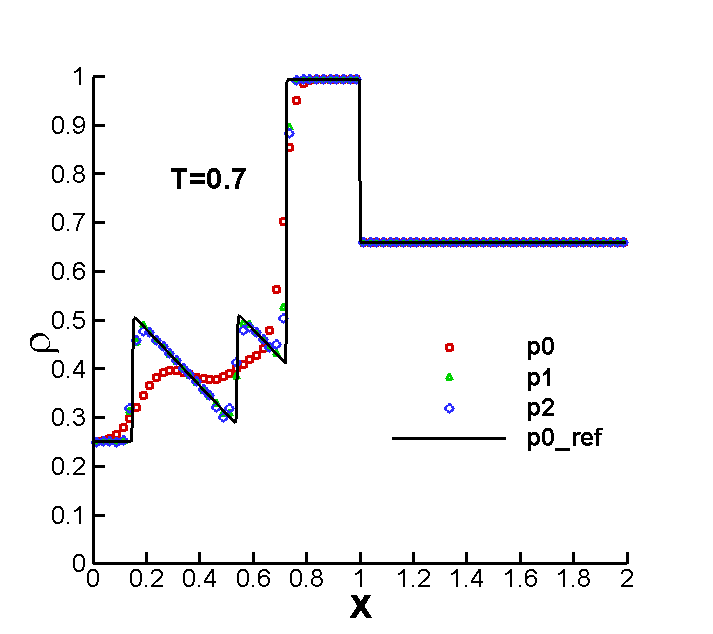}
\end{center}
\caption{
Bottleneck problem with initial and boundary data~(\ref{testb3}). 
}
\label{figtestb3}
\end{figure}

\subsection{Two incoming roads and one outgoing road}
Consider a crossing with two incoming roads and one outgoing road, all parametrized by $[0,1]$, with a fixed ``right of way parameter'' $q \in [0,1]$.
The incoming roads are denoted by 1 and 2, while the outgoing road is denoted by 3.
The flux function is given by the equation~(\ref{eq: flux_function}).

We test the scheme with the following initial and boundary data:
\begin{eqnarray}
&\rho_1(0,x)=
\left\{
\begin{array}{ll}
0.1, &\quad \mbox{if} \ x\in [0, 0.2] \cup [0.4, 0.6] \cup [0.8, 1], \\
0.2, &\quad \mbox{otherwise},
\end{array}
\right. \quad \rho_3(t=0,x)=0.1, \\ \notag
&\rho_2(t=0,x)=0.1+0.05 \sin(5 \pi x), \quad \rho_{1,b}(t,x=0)=0.1, \quad \rho_{2,b}(t,x=0)=0.1.
\label{twoone2}
\end{eqnarray}
We take $q=0.5$, see Figure~\ref{twoone_step}. Similar to the previous example, higher order schemes have better performance in resolving solution structures than lower order schemes.

 \begin{figure}[htb]
\begin{center}
\includegraphics[width=2.in]{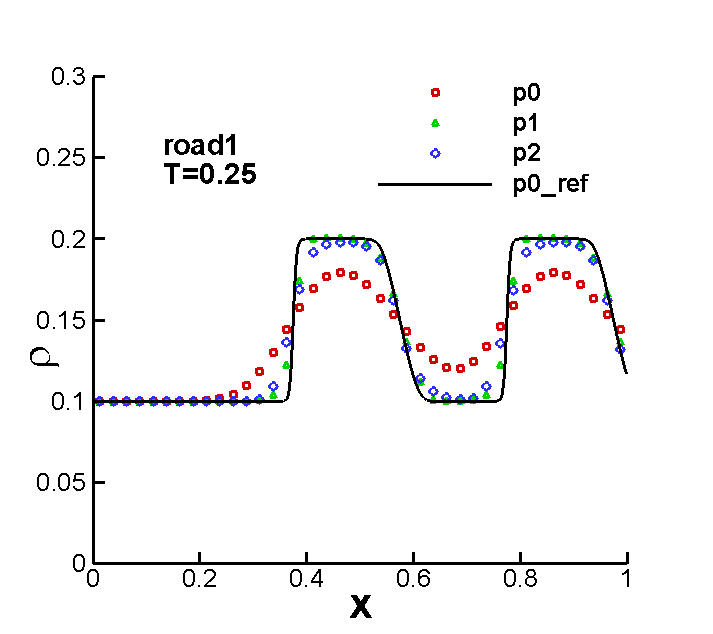}
\includegraphics[width=2.in]{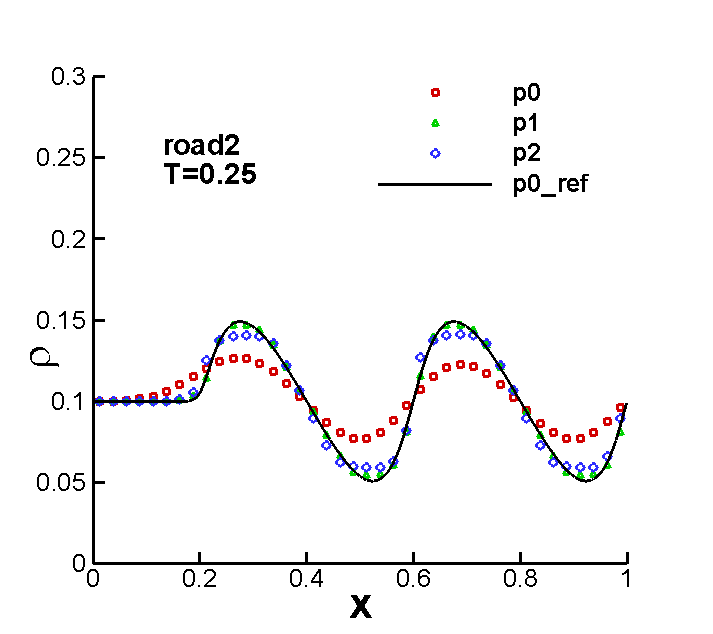}
\includegraphics[width=2.in]{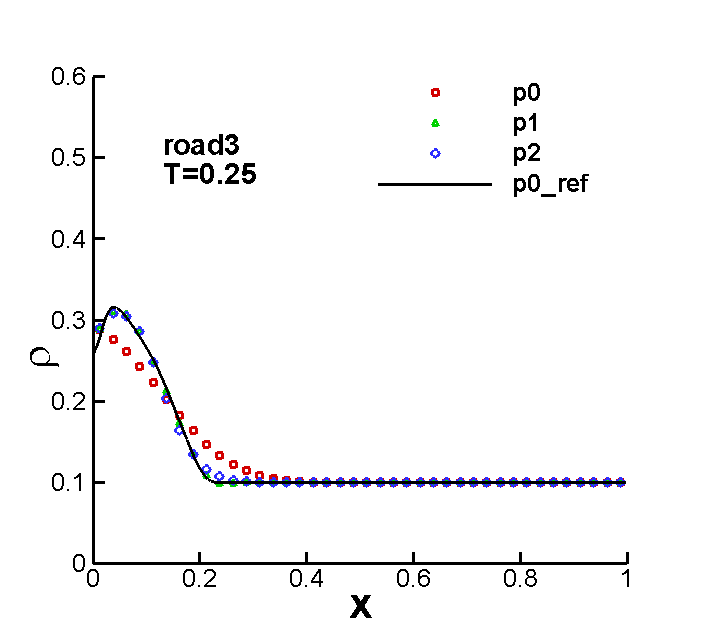}
\includegraphics[width=2.in]{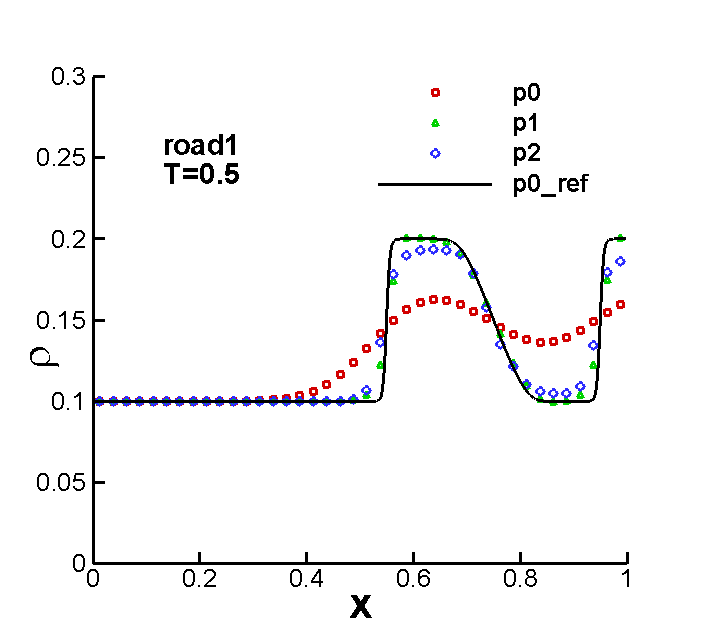}
\includegraphics[width=2.in]{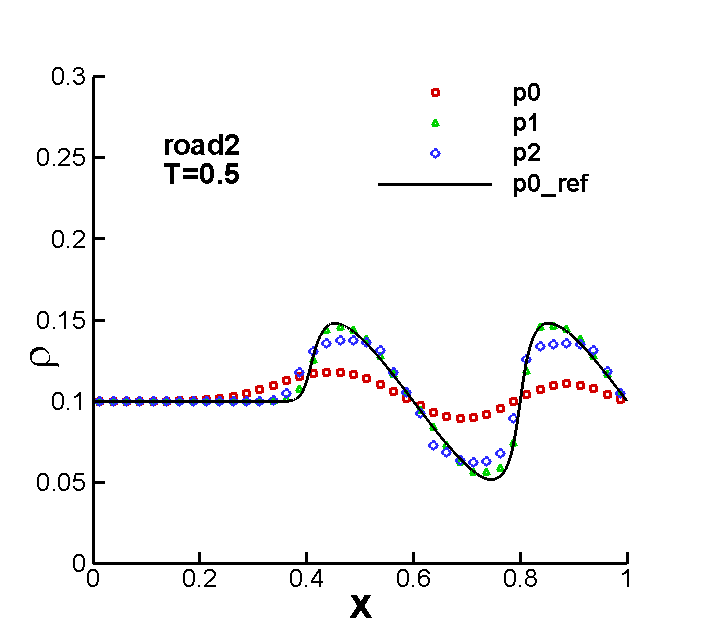}
\includegraphics[width=2.in]{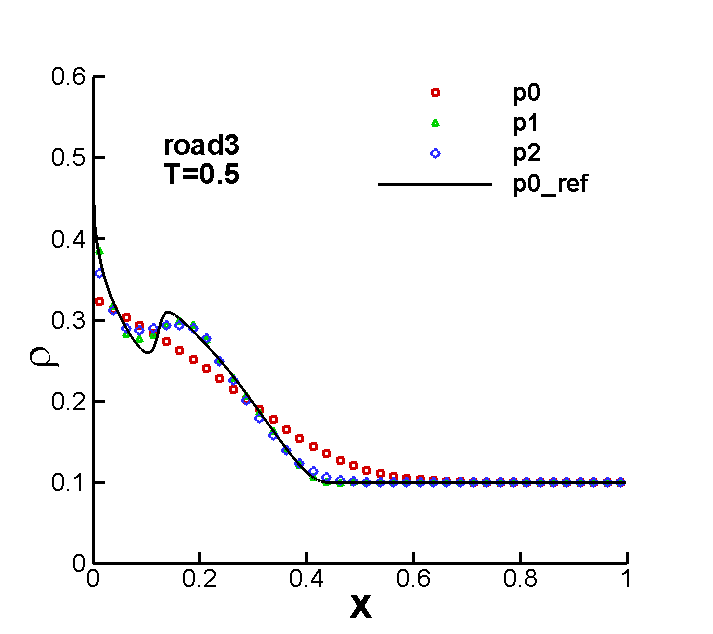}
\includegraphics[width=2. in]{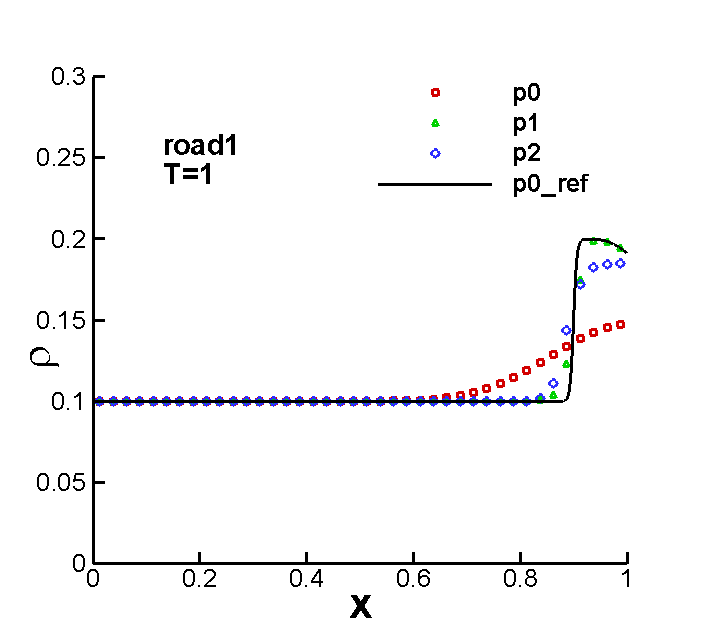}
\includegraphics[width=2. in]{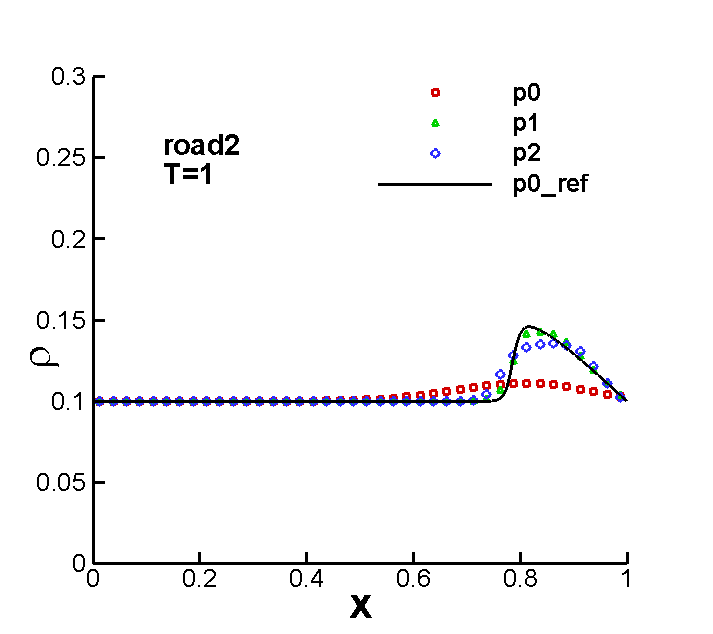}
\includegraphics[width=2. in]{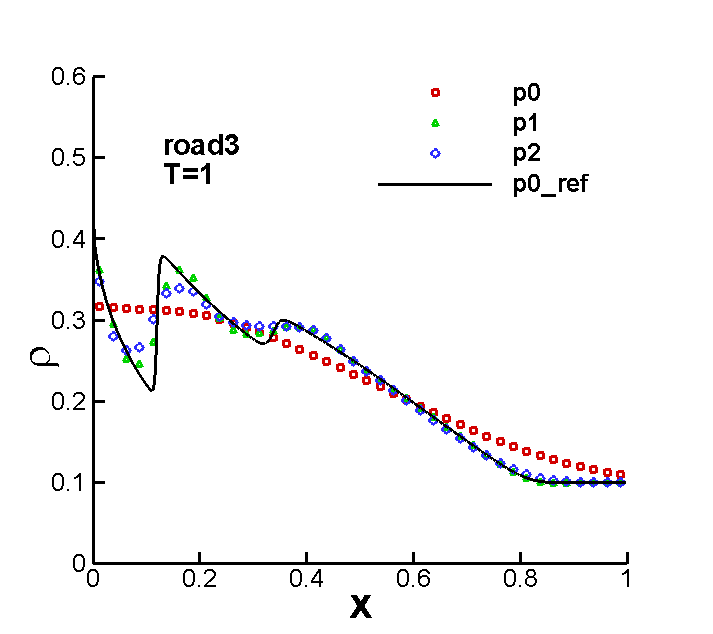}
\end{center} \caption{Two incoming and one outgoing roads with initial and boundary data~(\ref{twoone2}), $q=0.5$, $T=0.25, 0.5, 1$.}
\label{twoone_step}
\end{figure}

\subsection{Two incoming and two outgoing roads}
Here we consider the particular case of a junction with two incoming and two outgoing roads. The incoming roads are denoted by 1 and 2,
while the outgoing roads are 3 and 4.
Two incoming and two outgoing roads are parametrized by the interval $[0,1]$.
The flux function is given by equation~(\ref{eq: flux_function}).
The traffic distribution matrix is~(\ref{MatrixA}). Let $\alpha = 0.4$, $\beta = 0.3$ in our simulations.


We take the same constant initial and boundary data as in~\cite{bretti2006numerical},
\begin{eqnarray}
&\rho_2(0,x)=\rho_3(0,x)=\rho_{2,b}(t)=0.82732683535; \quad \rho_4(0,x)=0.5; \quad
\rho_{1,b}(t)=0.4;\\ \notag
&\rho_1(0,x)=\begin{cases}
						0.4, 	 	 &\quad \mbox{if}\  0\le  x \le 0.5, \\
						\rho_{1,0},  &\quad \mbox{otherwise}.	
					\end{cases}
\label{twotwo1}
\end{eqnarray}
In the panels shown in Figure~\ref{twotwo_t470}, we present numerical solutions on road 1 and road 3 at different times; the results are comparable to those produced in~\cite{bretti2006numerical}.
Higher order RKDG schemes are observed to have better performance than the first order scheme.

\begin{figure}[ht]
\begin{center}
\includegraphics[width=2.in]{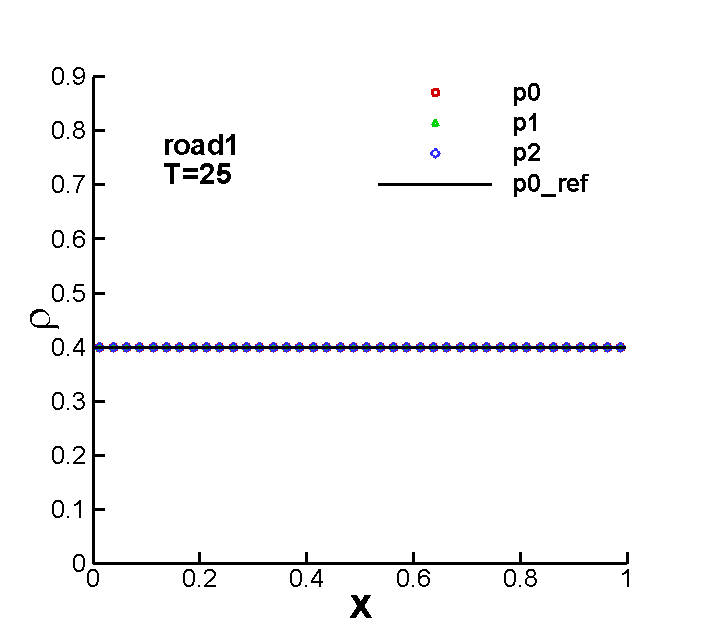}
\includegraphics[width=2.in]{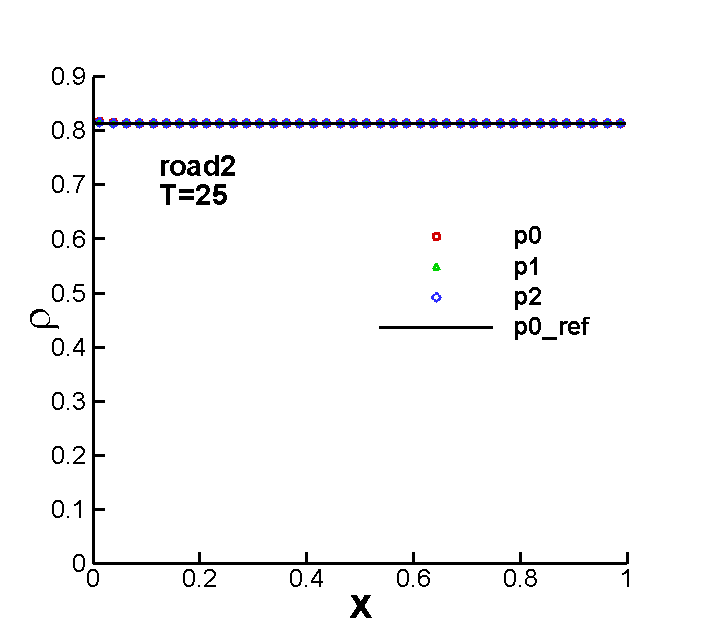}
\includegraphics[width=2.in]{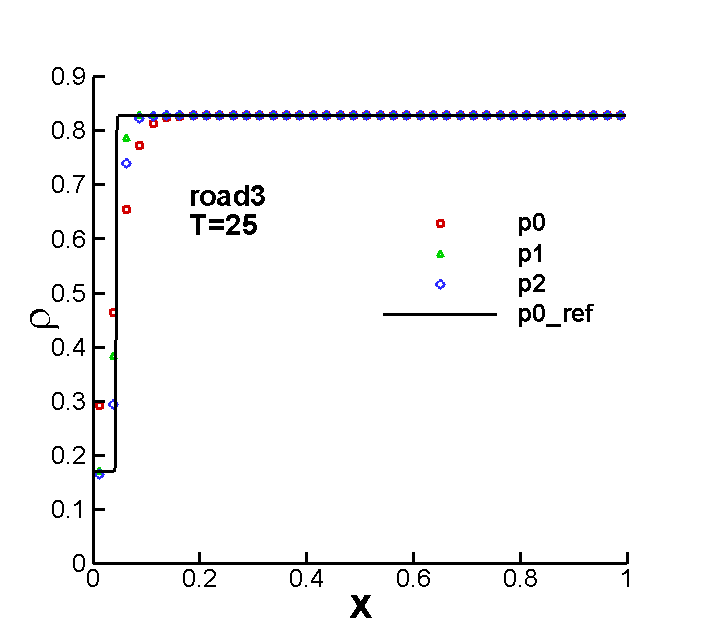}
\includegraphics[width=2.in]{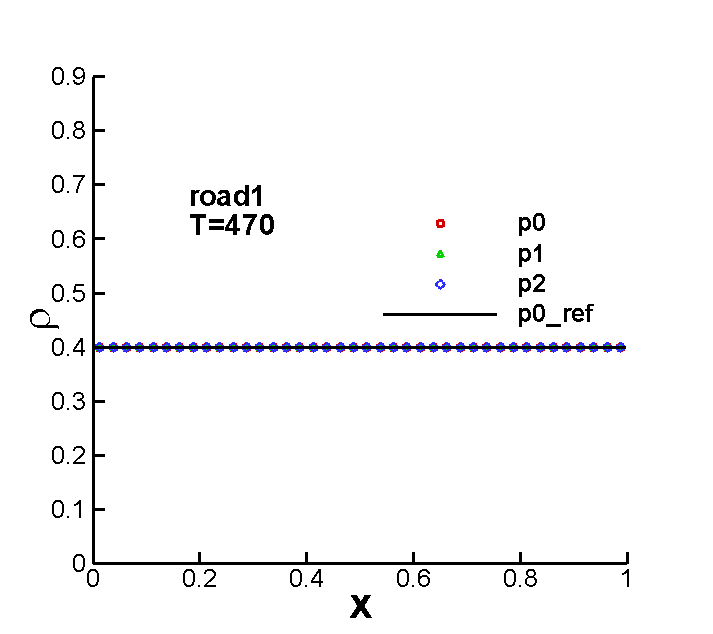}
\includegraphics[width=2.in]{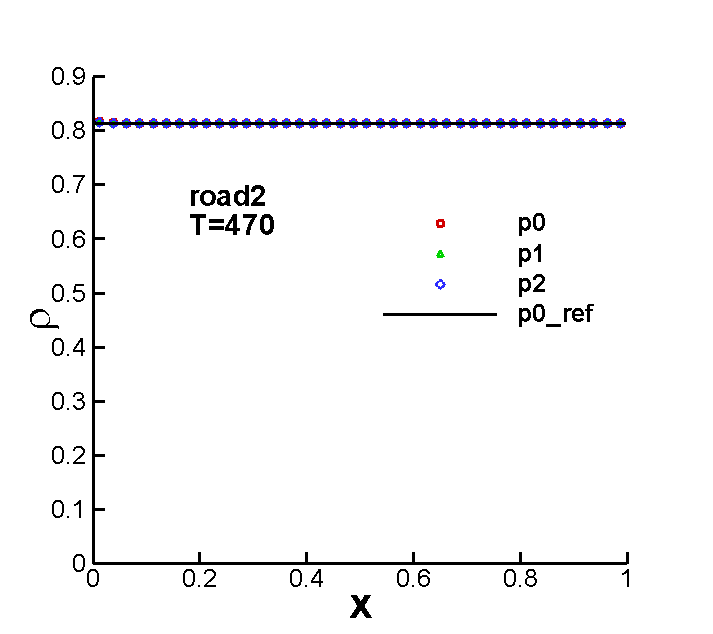}
\includegraphics[width=2.in]{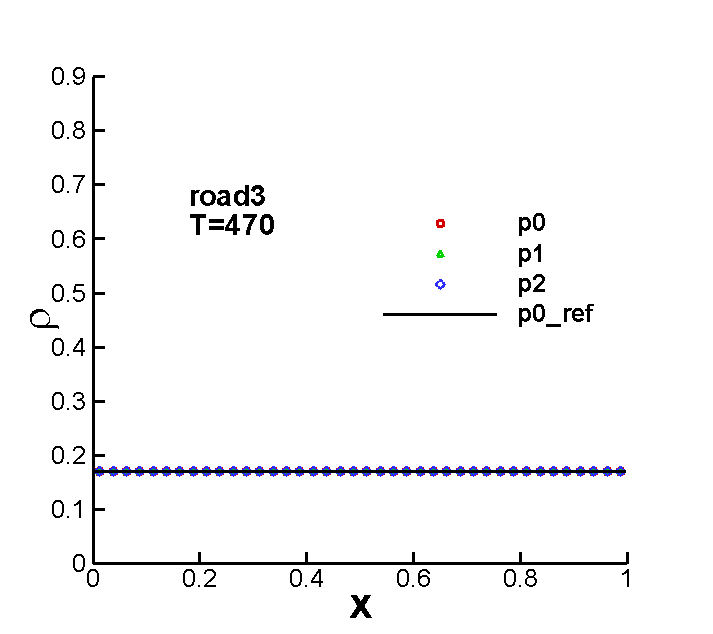}
\end{center} \caption{Two incoming and two outgoing roads with initial and boundary data~(\ref{twotwo1}), $T= 25$ (upper row) and $T=470$ (bottom row). 
}
\label{twotwo_t470}
\end{figure}

We then test another example with the following initial and boundary data:
\beq
\begin{split}
\rho_2(0,x)=0.2+0.1 \sin (5 \pi x); \quad \rho_3(0,x)=\rho_4(0,x)=0.5; \quad
\rho_{1,b}(t)=\rho_{2,b}(t)=0.2;\\
\rho_1(0,x)=\begin{cases}
						0.2, &\quad \mbox{if} \ x \in [0, 0.2] \cup [0.4, 0.6] \cup [0.8, 1],  \\
						0.4, &\quad \mbox{otherwise}.	
					\end{cases}
\end{split}
\label{twotwo2}
\eeq
As can be seen from Figure~\ref{twotwo_step}, RKDG methods with $P^1$ and $P^2$ solution spaces approximate reference solution very well, compared with that from the first order scheme, which has been greatly smeared due to numerical diffusions.
\begin{figure}[ht]
\begin{center}
\includegraphics[width=2.in]{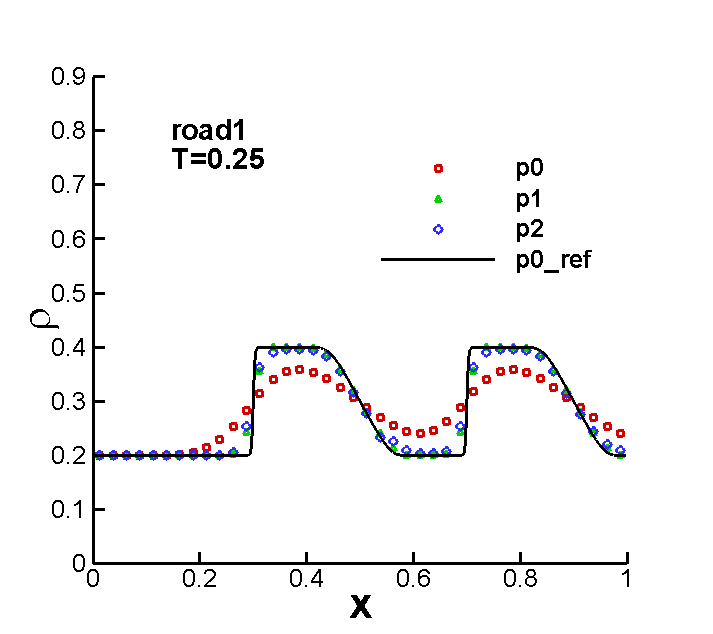}
\includegraphics[width=2.in]{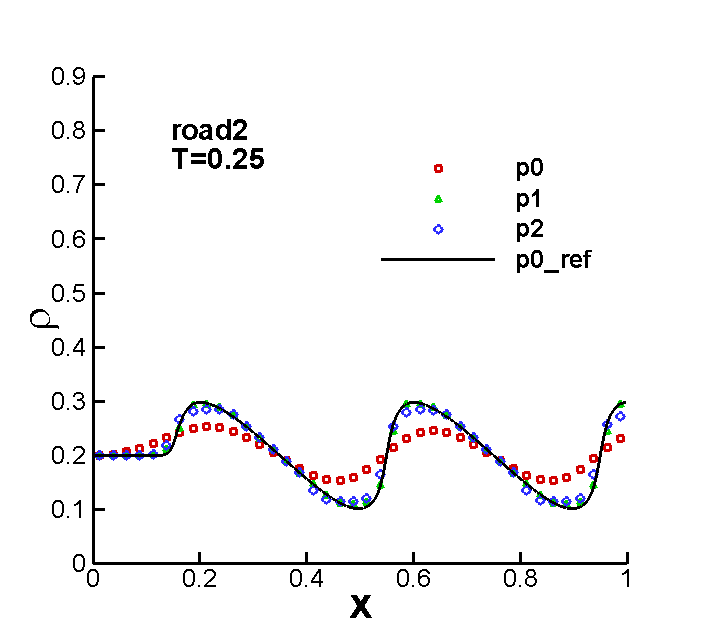}
\includegraphics[width=2.in]{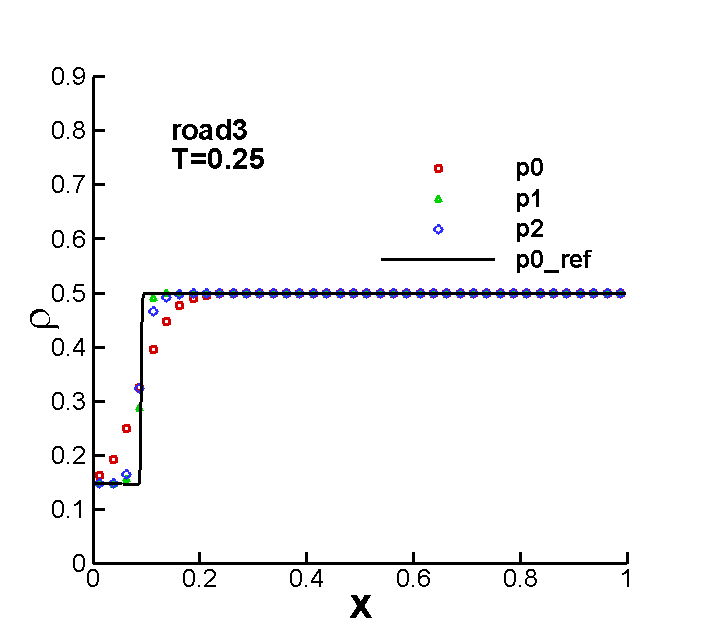}
\includegraphics[width=2.in]{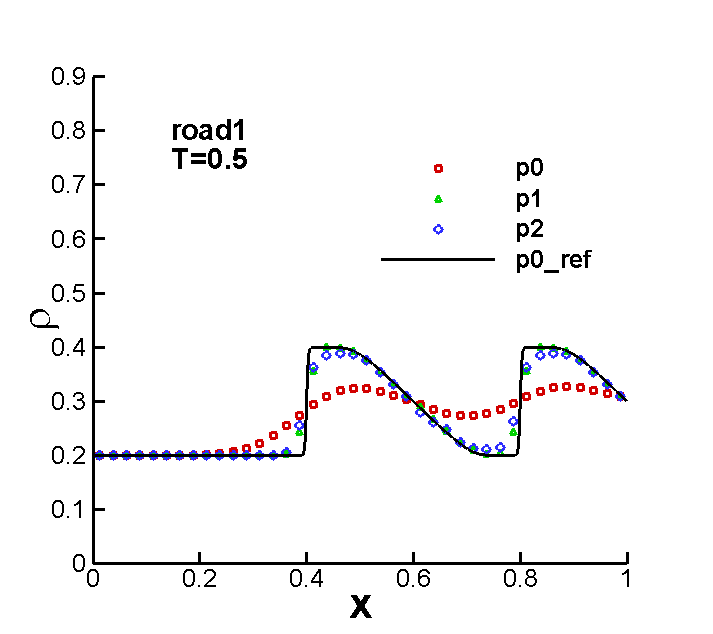}
\includegraphics[width=2.in]{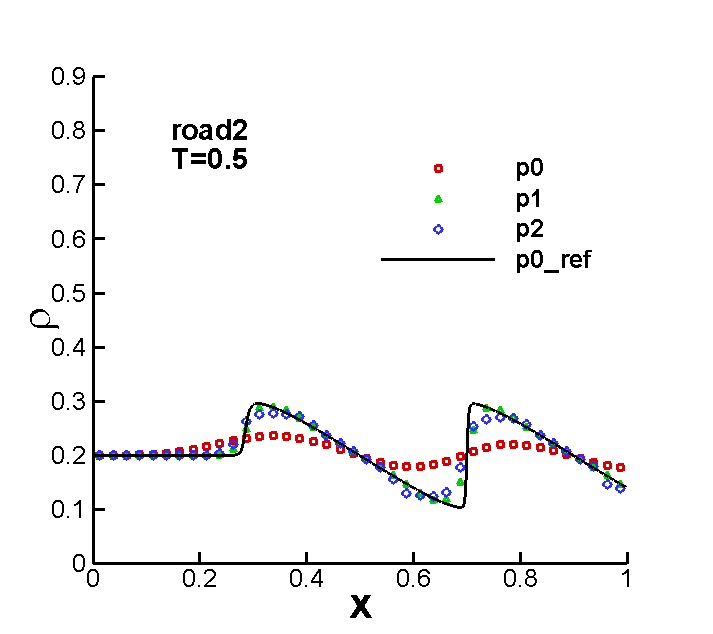}
\includegraphics[width=2.in]{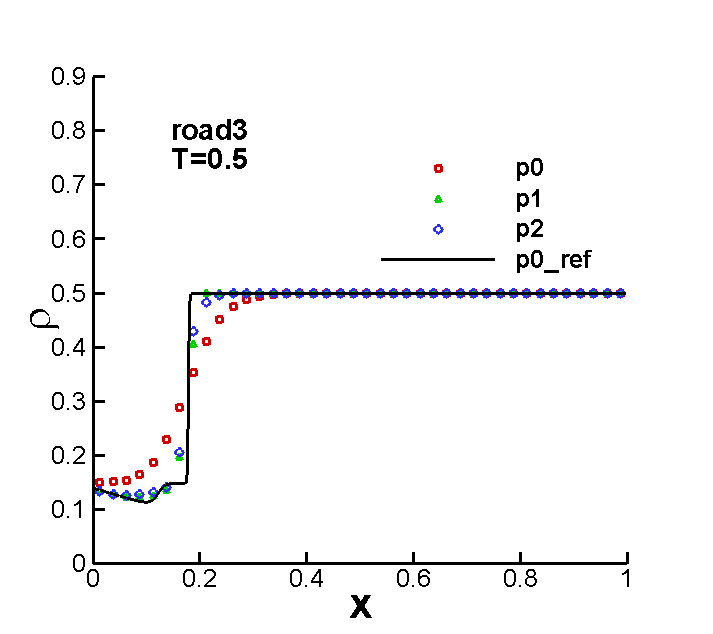}
\includegraphics[width=2.in]{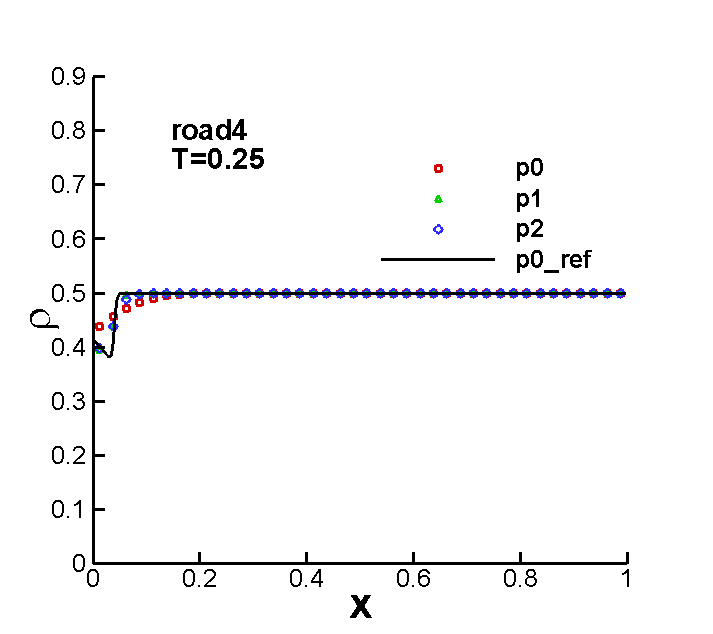}
\includegraphics[width=2.in]{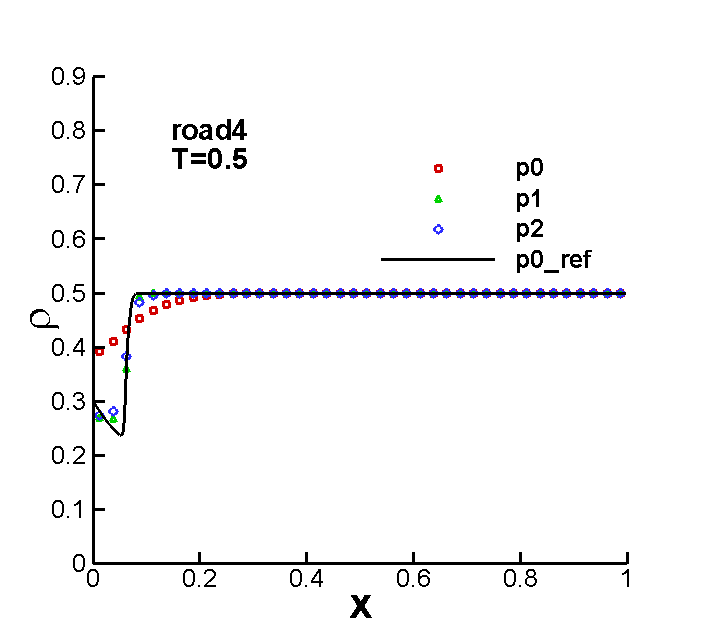}
\end{center} \caption{Two incoming and two outgoing roads with initial and boundary data~(\ref{twotwo2}), $T=0.25, 0.5$. 
}
\label{twotwo_step}
\end{figure}

\subsection{Traffic Circles}
In this part we present some simulations reproducing a simple traffic circle composed of 8 roads and 4 junctions, as shown in Figure~\ref{trafficcircle}. Consider the following initial and boundary data:
\beq
\begin{split}
\rho_3(0,x)=\rho_4(0,x)=\rho_{1R}(0,x)=\rho_{2R}(0,x)=\rho_{3R}(0,x)=\rho_{4R}(0,x)=0.5; \\
\rho_{1,b}(t)=0.25; \quad \rho_{2,b}(t)=0.4; \quad \rho_2(0,x)=0.2+0.2 \sin(5 \pi x);\\
\rho_1(0,x)=\begin{cases}
						0.25, &\quad	 \mbox{if} \ x\in [0, 0.2] \cup [0.4, 0.6] \cup [0.8, 1], \\
						0.35, &\quad	 \mbox{otherwise}.
					\end{cases}
\end{split}
\label{traffic_cir}
\eeq

The distribution coefficients, namely $(\alpha_{1R,3}, \alpha_{1R,2R}, \alpha_{3R,4}, \alpha_{3R,4R})$ are assumed to be constant and are all equal to $\alpha=0.5$. Let us choose the following priority
parameters, with $q_1=q(1,4R,1R)=0.25$,  $q_2=q(2,2R,3R)=0.25$. The fixed values imply that road $4R$ is the through street with respect to road 1, and road $2R$ is the through street with respect to road 2.
The roads 1, 4R; 1R and 2, 2R; 3R are two incoming and one outgoing roads, while the roads 1R; 2R, 3 and 3R; 4R, 4 are one incoming and two outgoing roads respectively.
In Figures~\ref{trafficcircle_t05},~\ref{trafficcircle_t1}, we present numerical solutions on all roads at $T=0.5, 1$. Higher order RKDG schemes are observed to have better performance than the first order scheme.

\begin{figure}[ht]
\begin{center}
\includegraphics[width=3.in]{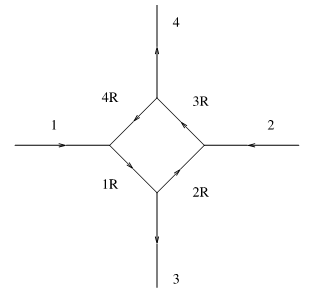}
\end{center} \caption{Traffic circle}
\label{trafficcircle}
\end{figure}


 \begin{figure}[ht]
\begin{center}
\includegraphics[width=2.in]{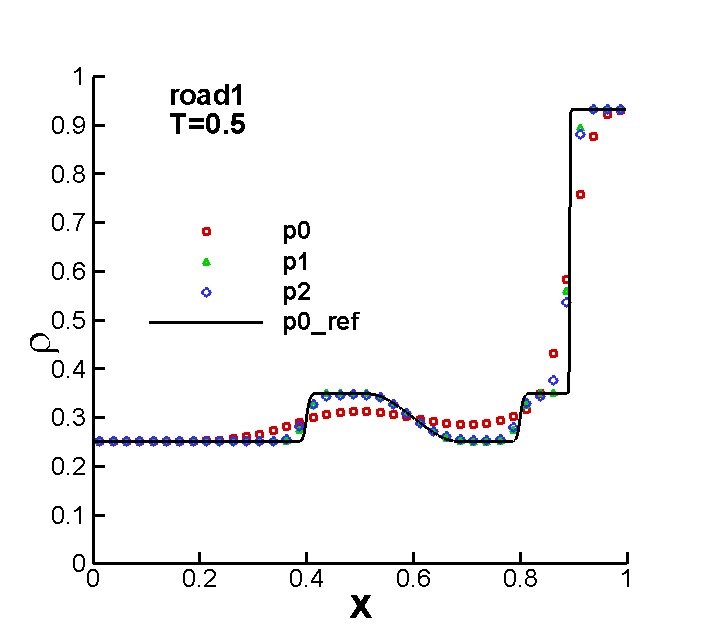}
\includegraphics[width=2.in]{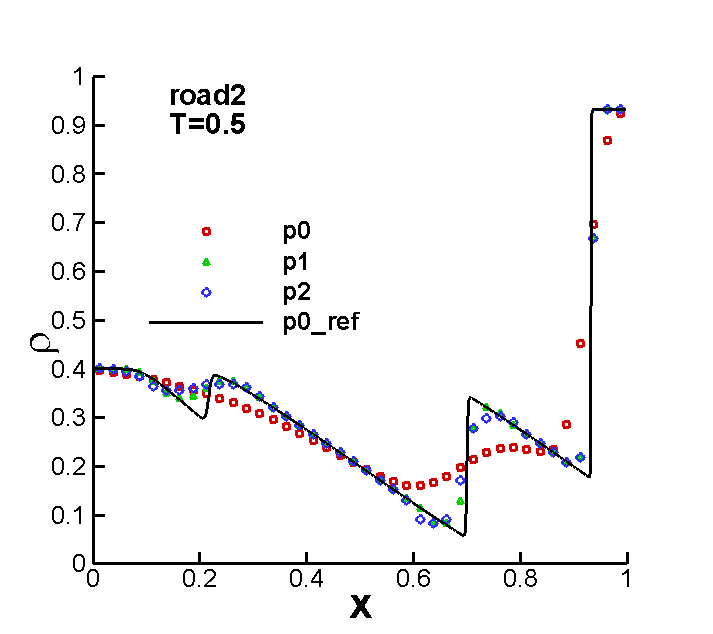}
\includegraphics[width=2.in]{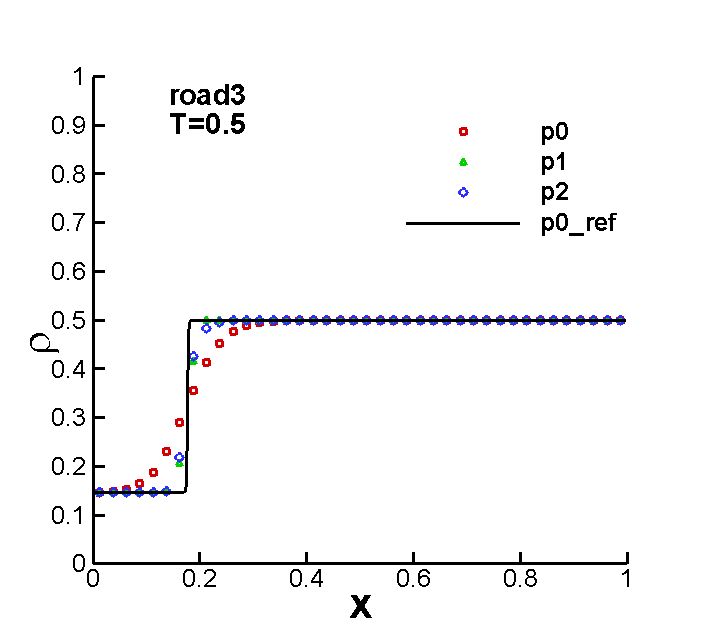}
\includegraphics[width=2.in]{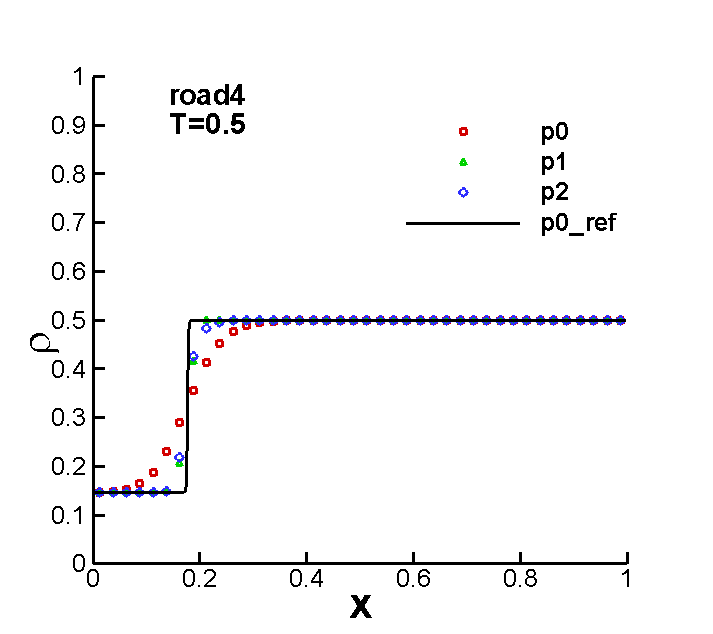}
\includegraphics[width=2.in]{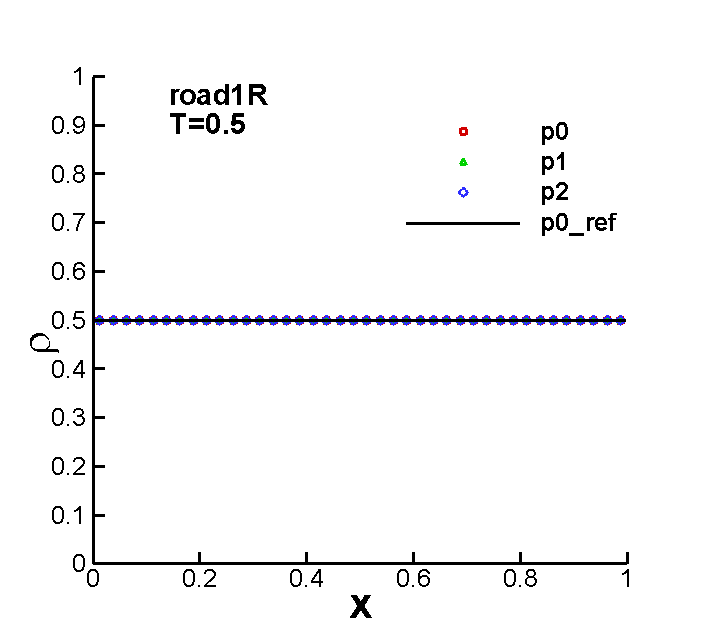}
\includegraphics[width=2.in]{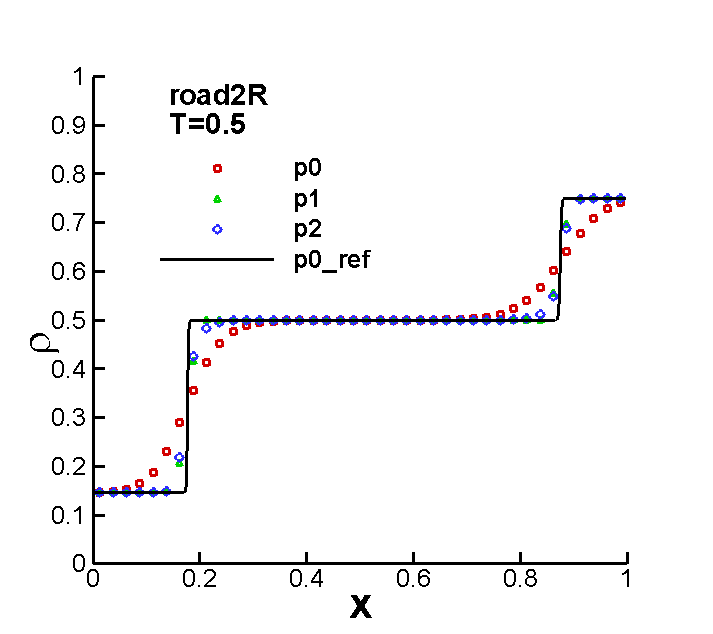}
\includegraphics[width=2.in]{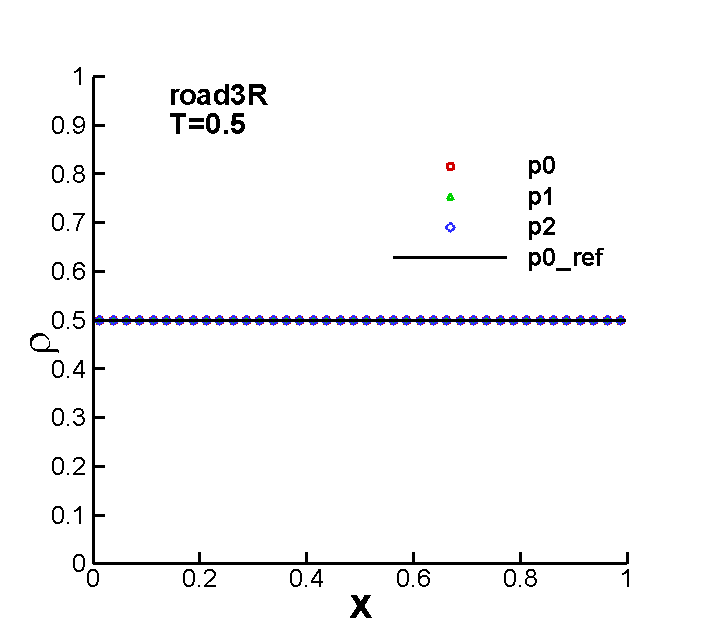}
\includegraphics[width=2.in]{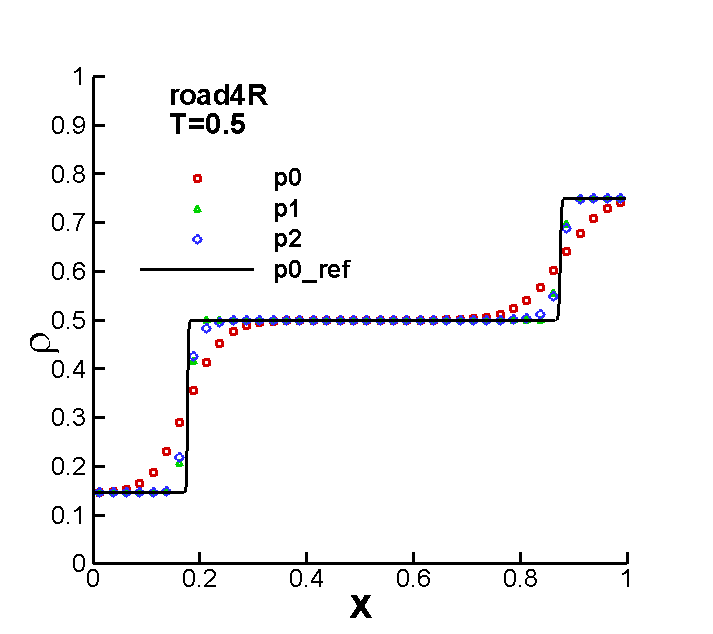}
\end{center} \caption{Traffic circle with initial and boundary data~(\ref{traffic_cir}), $q_1=q_2=0.25$, T=0.5, road 1, 2, 3, 4, 1R, 2R, 3R, 4R.}
\label{trafficcircle_t05}
\end{figure}

 \begin{figure}[ht]
\begin{center}
\includegraphics[width=2.in]{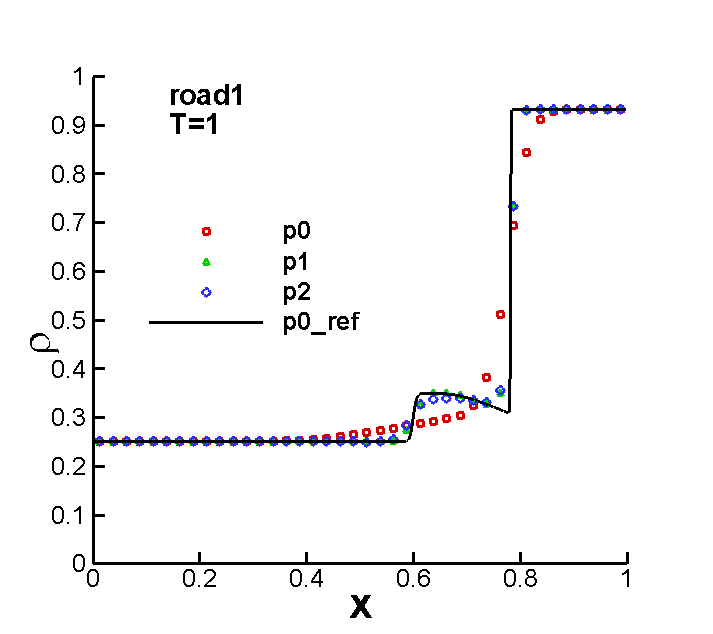}
\includegraphics[width=2.in]{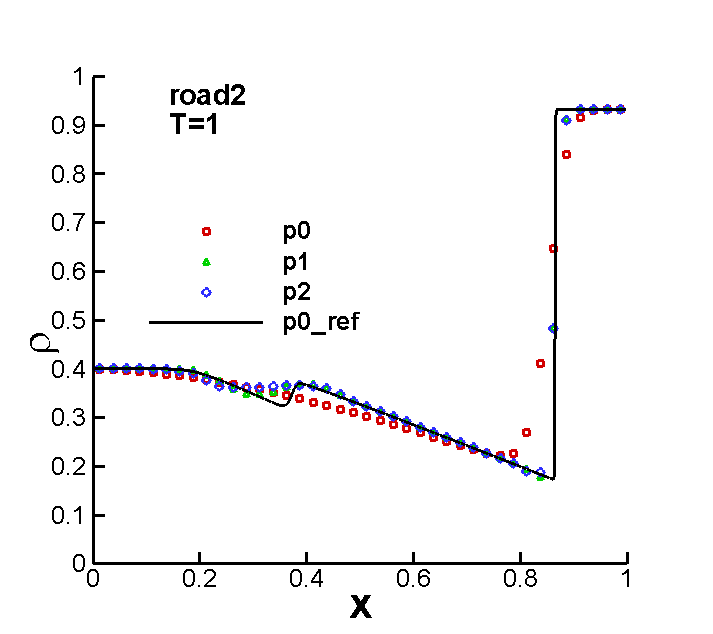}
\includegraphics[width=2.in]{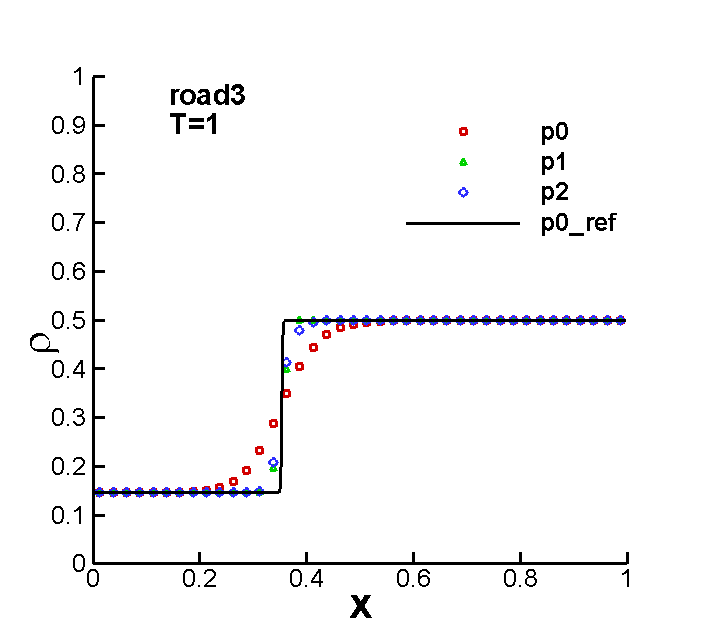}
\includegraphics[width=2.in]{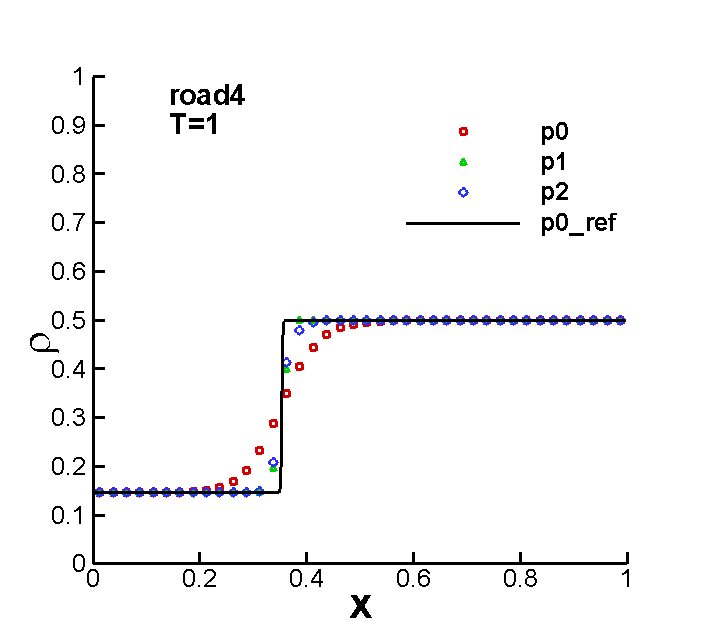}
\includegraphics[width=2.in]{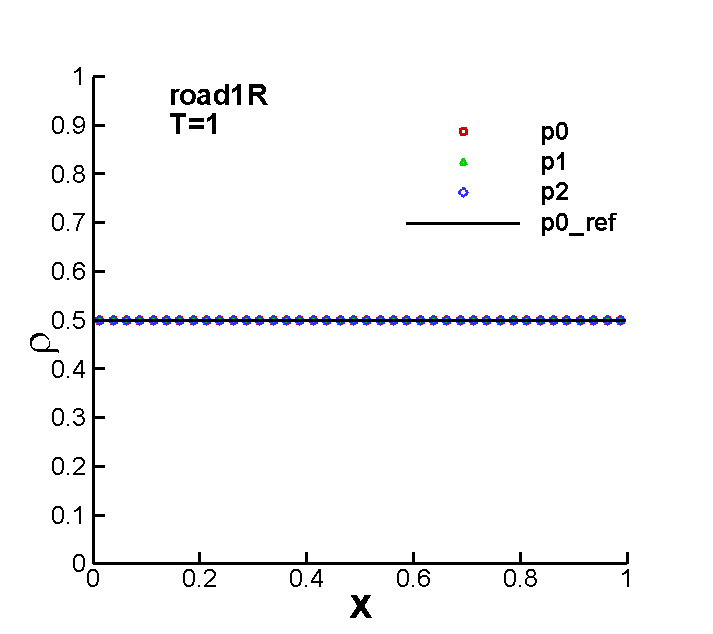}
\includegraphics[width=2.in]{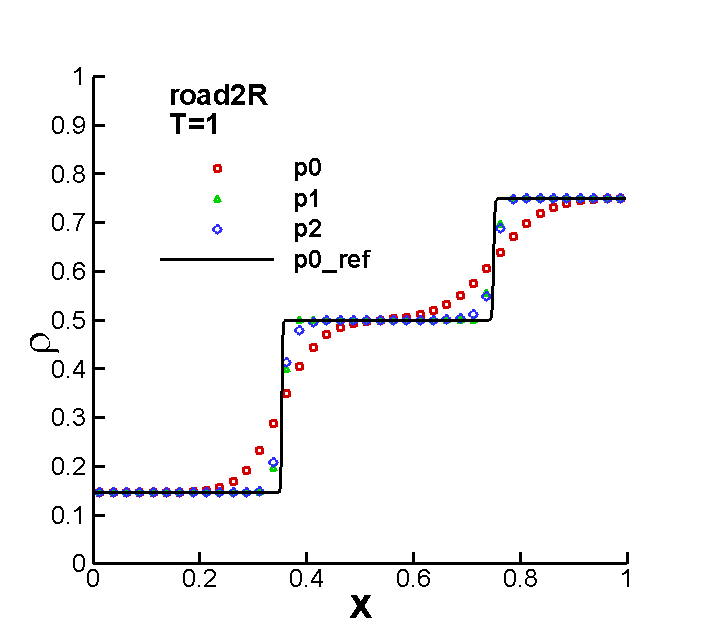}
\includegraphics[width=2.in]{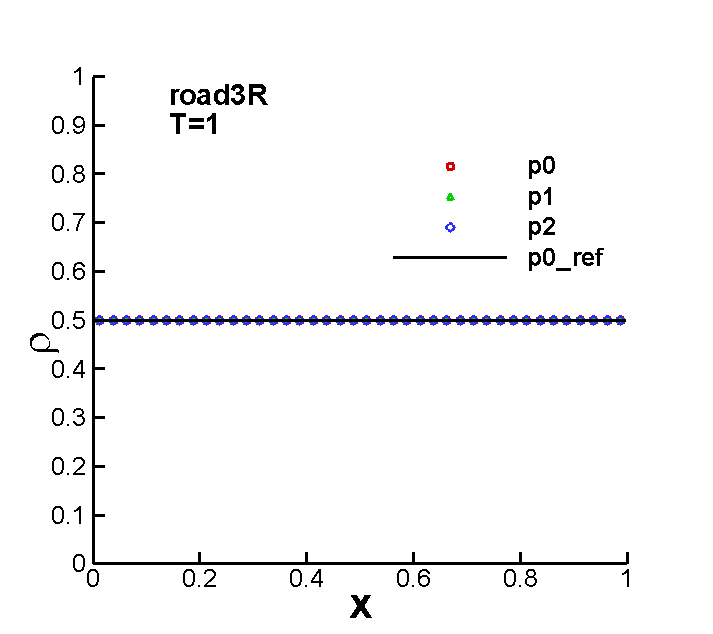}
\includegraphics[width=2.in]{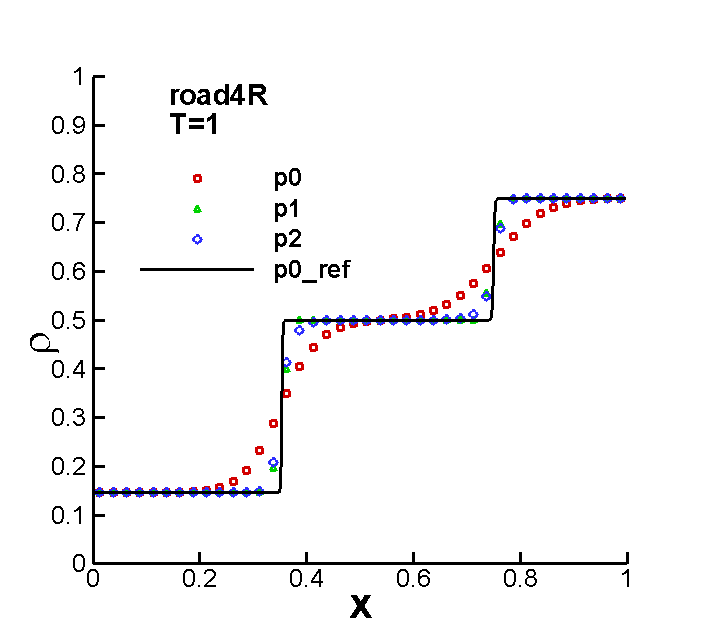}
\end{center} \caption{Traffic circle with initial and boundary data~(\ref{traffic_cir}), $q_1=q_2=0.25$, T=1, road 1, 2, 3, 4, 1R, 2R, 3R, 4R.}
\label{trafficcircle_t1}
\end{figure} 

\section{Conclusion}
\label{conclusion}

{
In this paper, we proposed a bound-preserving, high order RKDG method for hyperbolic network problems with traffic flow applications. 
Compared with other existing higher order methods, DG methods are compact in the sense that only direct neighbors are used to update
the solution on one element. 
Such a property offers great convenience in handling boundary conditions at junctions with high order accuracy. 
This was demonstrated on several examples, including those involving solutions with rich solution structures, where 
higher-order accuracy at edges and vertices of the traffic flow network provided superior solution quality in comparison with 
the classical first-order methods, while keeping low computational complexity when dealing with coupling conditions at vertices.  
Extensions of the proposed algorithm to {\sl systems} of hyperbolic conservation laws on networks,
such as those, e.g., in \cite{CanicStent1,CanicStent2}, are subject to future investigations. 
}

\bibliographystyle{siam}
\bibliography{refer}

\end{document}